\documentclass[10pt,fleqn]{article}
\usepackage[paperwidth=192mm,paperheight=262mm,left=13.5mm,right=13.5mm,
 top=19mm,bottom=18mm,headheight=12pt,headsep=7pt,footskip=18pt]{geometry}
\usepackage[T1]{fontenc}
\usepackage{amsmath,amssymb,amsfonts,mathtools,amsthm}
\usepackage{txfonts}
\usepackage{booktabs,array,enumitem,microtype,needspace,float,mathrsfs}
\usepackage[numbers,sort&compress]{natbib}
\usepackage{etoolbox,xcolor,url,fancyhdr}
\usepackage[hyperfootnotes=false,hypertexnames=false,hidelinks]{hyperref}
\hypersetup{bookmarksnumbered=true,bookmarksdepth=2,
 pdfauthor={Yangcheng Li},
 pdftitle={Reconstruction of Binary Linear Systems and Profile Geometry of Sparse Krylov Strata}}
\numberwithin{equation}{section}
\AtBeginEnvironment{theorem}{\Needspace{6\baselineskip}}
\AtBeginEnvironment{proposition}{\Needspace{5\baselineskip}}
\AtBeginEnvironment{corollary}{\Needspace{5\baselineskip}}
\AtBeginEnvironment{lemma}{\Needspace{5\baselineskip}}
\AtBeginEnvironment{example}{\Needspace{5\baselineskip}}
\newtheorem{theorem}{Theorem}[section]
\newtheorem{proposition}[theorem]{Proposition}
\newtheorem{lemma}[theorem]{Lemma}
\newtheorem{corollary}[theorem]{Corollary}

\newtheorem{remark}[theorem]{Remark}
\newtheorem{example}[theorem]{Example}

\newcommand{\A}{\mathbb A}
\newcommand{\PP}{\mathbb P}
\newcommand{\Gm}{\mathbb G_{\mathrm m}}
\newcommand{\cF}{\mathcal F}
\newcommand{\cE}{\mathcal E}
\newcommand{\Gr}{\operatorname{Gr}}

\newcommand{\rank}{\operatorname{rank}}
\newcommand{\Hilb}{\operatorname{Hilb}}
\newcommand{\rwt}{\operatorname{rwt}}

\newcommand{\sym}{\operatorname{Sym}}
\newcommand{\Fact}{\operatorname{Fact}}
\newcommand{\Disc}{\operatorname{Disc}}
\newcommand{\Res}{\operatorname{Res}}

\newcommand{\Irr}{\operatorname{Irr}}

\begin{document}
\title{Reconstruction of Binary Linear Systems and Profile Geometry of Sparse Krylov Strata}
\author{Yangcheng Li\thanks{Corresponding author. Email: \texttt{liyc@m.scnu.edu.cn}.}\\[3pt]
\small School of Mathematical Sciences, South China Normal University\\
\small Guangzhou 510631, Guangdong, China}
\date{}
\maketitle
\begin{abstract}
We study reconstruction and profile geometry for divisor schemes of binary
linear systems and their sparse Krylov charts.  Over the integers, the first
nonzero equations of the complete embedded divisor scheme recover the
defining linear system functorially under arbitrary base change, yielding a
closed immersion from the Grassmannian of linear systems to the corresponding
Hilbert scheme.  For monomial systems in characteristic zero, arithmetic
profiles classify the reduced factorization branches, determine their image
dimensions and generic multiplicities, and control geometric reducedness.
For complete progressions, the normalizations of the relation branches and
their images are explicit products of projective spaces equipped with two
natural polarizations.  We give an affine normality criterion in terms of the
associated Fourier data and a projective criterion obtained by adjoining an
endpoint-allocation condition.  These results place the closed sparse
Krylov rank loci in a uniform reconstruction--normalization framework and
yield explicit mixed-degree and formal-profile consequences.  They also
clarify the limit of normalization data alone: recovering a possibly
nonnormal image algebra requires additional information not addressed here.
\end{abstract}


\section{Introduction and main results}\label{sec:introduction}

This paper studies the reconstruction and profile geometry of sparse divisor
loci arising from binary linear systems.  The basic affine chart is the
remainder--Krylov rank locus attached to a monomial relation space: if
\(g\) is monic of degree \(r\), \(T_g\) is multiplication by \(x\) on
\(K[x]/(g)\), and \(Q_i\) is the coefficient vector of \(x^i\bmod g\), then
\[
 g\mid \sum_{i\in I}c_i x^i
 \quad\Longleftrightarrow\quad
 \sum_{i\in I}c_iQ_i=0.
\]
Thus the same relation can be viewed either as a coefficient-space rank
condition or as a factorization condition in a varying binary linear system.
Two related papers \cite{LiYuanOrbits,LiYuanSchur} use the same remainder
coordinates to study, respectively, the rational-normal-curve/GRS locus and
the MDS open locus.  The present paper addresses a different object: the
closed rank strata attached to arbitrary binary relation spaces.

The first question is reconstruction.  Let \(V\subset H^0(\PP^1,\mathcal
O(N))\) be a rank-\(s\) binary linear system and let \(X_r(V)\subset\PP^r\)
be the determinantal scheme of degree-\(r\) forms dividing a nonzero member
of \(V\).  Although every such system has the same determinantal postulation,
its first nonzero equation space remembers the system.  The integral
maximal-minor map is an isomorphism, and contraction recovers \(V\)
functorially after arbitrary base change.  In particular, the induced map
from the Grassmannian of linear systems to the Hilbert scheme of embedded
divisor schemes is a closed immersion.  This is a family-level statement: it
retains the fixed coefficient embedding, allows base points and fixed roots,
and does not reduce to generic pointwise identifiability.

The second question concerns the geometry of the sparse monomial charts.  In
characteristic zero, after removing the forced root at zero and the common
power in the exponent support, geometric monodromy is a wreath product.  Its
orbits on selected root subsets are encoded by arithmetic necklace profiles.
We show that these profiles classify the reduced incidence components,
determine which components dominate coefficient-space rank components, and
compute their generic scheme multiplicities.  Occupancy gives the generic
evaluation rank and hence the dimension of the projected image.  This yields,
in particular, an exact geometric reducedness criterion for the fixed-support
rank locus.

The third question is normalization.  For an arbitrary sparse support we
classify the normal relation covers by their wreath-product stabilizers.  For
complete progressions, both the relation-branch normalization and the image
normalization become explicit products of projective spaces.  The two
natural hyperplane classes pull back with coefficients determined by the
necklace weights and stabilizers.  These formulas give normality criteria for the affine monic chart and for
the projective image: the monic image is normal precisely when the weighted
generator degrees do not collide and the relevant Fourier coefficients are
nonzero, while normality of the projective image requires in addition
uniqueness of the endpoint allocation.  A prime-step corollary gives a particularly simple classification.

The normalization results mark a natural boundary of the present paper.
Once the normal profile image and its polarizations are known, a finer
question is to recover the actual, possibly nonnormal, image algebra from its
normalization.  Such a recovery requires information beyond normalization and
its natural polarizations and is not pursued here.  We keep the focus on
reconstruction, profile geometry, normalization, and the degree and
formal-local consequences that follow from them.

\subsection{Principal results}

The reconstruction theorem is obtained from binary multiplication and an
integral maximal-minor isomorphism.  If \(L=N-r+1\), the first nonzero
equations of \(X_r(V)\) occur in degree \(L\); after the maximal-minor
identification they are exactly \(\bigwedge^L V^\perp\).  Theorem
\ref{v6:lin:reconstruction} recovers \(V\) by contraction and proves that the
Hilbert morphism is a closed immersion.  The construction commutes with
arbitrary base change.

For a monomial support, Theorem \ref{thm:dominant-profile-components}
classifies the reduced factorization branches by admissible arithmetic
profiles.  It gives their reduced covering degrees, their scheme-theoretic
generic multiplicities in the full incidence, their image dimensions, and
the generic projective fibers of the forgetful map.  Corollary
\ref{cor:geometric-reducedness} converts these data into the reducedness
criterion.

The normalization theory begins with Proposition \ref{v6:norm:cover}.  Under
the complete-progression hypothesis, Theorem \ref{v6:norm:product} identifies
the two normalizations as products of projective spaces and computes the two
polarizations.  Theorems \ref{v6:norm:affine-criterion} and
\ref{v6:norm:projective-criterion} give the exact normality criteria for the affine monic chart and the projective
image.  The product formulas also underlie the mixed-degree
calculations and formal-profile models recorded in the appendices.

\subsection{Relation with precedents and scope}

The determinantal and secant-bundle foundations are classical, and the
multiplication-induced Hermite map and its reciprocity formalism provide a
neighboring viewpoint.  Three distinctions are important.  First, we do not
claim a new Hermite reciprocity isomorphism.  Second, in the hypersurface case
\(s=r\), the first-equation line is the classical
Pl\"ucker--Hermite/Poncelet construction; see, for example,
\cite[Section~3.2]{IlardiSupinoValles}.  Third, the common determinantal
postulation of the divisor schemes is itself standard.  What is proved here
for every \(s\le r\) is that the complete first nonzero equation bundle of
the embedded divisor scheme is the transported exterior-power subbundle,
that contraction recovers the defining linear system functorially after
arbitrary base change, and that the resulting Grassmannian-to-Hilbert
morphism is a closed immersion.  The sparse monodromy input is the
wreath-product theorem of Esterov--Lang \cite{EsterovLang}; the profile
results use it to track the selected factor, contractions under the image
projection, and scheme-theoretic generic multiplicities.

The product-normalization results should also be distinguished from the
finer problem of recovering the actual image algebra.  Knowing the
normalization, even together with the natural polarizations, does not by
itself determine the actual image.  In the present paper normalization is the
endpoint of the main argument.

\subsection{Organization}

Section~\ref{sec:remainders-new} fixes remainder coordinates and the divisor
scheme.  Section~\ref{sec:reconstruction-new} proves integral reconstruction.
Section~\ref{sec:projective-incidence} develops the relation incidence and
arithmetic profiles.  Section~\ref{sec:normalization} classifies normal
relation covers and constructs the product normalizations.  Section
\ref{sec:normality-new} gives the normality criteria and the prime-step
consequence.  Section~\ref{sec:conclusion-new} concludes with further directions and
records the scope and limitations of the normalization results.  The
appendices collect the fixed-support, profile-degree, formal-local, and
selected auxiliary results used to support the statements developed here.
\section{Remainder coordinates and the divisor-scheme setup}\label{sec:remainders-new}\label{sec:preliminaries}

The remainder coordinates of \cite{LiYuanOrbits,LiYuanSchur} identify
polynomial relations with divisibility.  Binary multiplication then
extends this affine model to arbitrary linear systems.

\paragraph{Conventions.}
Throughout, $K$ is a field, $\overline K$ is an algebraic closure, and $1\le r<n$.
A geometric property means the corresponding property after base change to $\overline K$.
We write $X_{\rm red}$ for the reduced induced scheme and $|X|$ for its underlying
support.  General geometric points are taken in a nonempty open subset after geometric
base change.  The monodromy and profile results are in characteristic zero unless
explicitly stated otherwise.
According to context, $\Fact_r(F/R)$ denotes the factorization algebra or its associated
affine $R$-scheme.
We set $\binom uv=0$ for $u\ge0$ whenever $v<0$ or $v>u$.
The symbols $D_I$ and $D(\boldsymbol m)$ denote respectively a root-weight
sum and a profile orbit degree.  We reserve $\eta,\xi$ for the two incidence hyperplane classes, and a
generic point is denoted by $\zeta_Y$.

\subsection{Remainder columns and the universal coefficient scheme}

For $g=x^r+a_{r-1}x^{r-1}+\cdots+a_0\in K[x]$, write
\begin{equation}\label{eq:intro-H}
 x^i\bmod g=\sum_{b=0}^{r-1}q_{b,i}x^b,\qquad
 Q_i=(q_{0,i},\ldots,q_{r-1,i})^{\mathsf T},\qquad
 H_g^{\rm rem}=(Q_0\ Q_1\ \cdots\ Q_{n-1}).
\end{equation}
For any polynomial $f=\sum_i f_ix^i$, the vector $\sum_i f_iQ_i$ is its remainder
modulo $g$; it is zero exactly when $g\mid f$.  Equivalently, for a cyclic pair
$(T,z)$ with minimal polynomial $g$, one has $f(T)z=0$ if and only if $g\mid f$.

Let $T_g$ be multiplication by $x$ on $K[x]/(g)$ in the power basis
$1,x,\ldots,x^{r-1}$.  Then
\begin{equation}\label{eq:remainder-recurrence}
 Q_i=T_g^iQ_0,\qquad Q_0=e_0,\qquad Q_i=e_i\ (0\le i<r),
\end{equation}
and
\begin{equation}\label{eq:Q-recurrence}
 Q_{i+r}=-a_{r-1}Q_{i+r-1}-\cdots-a_1Q_{i+1}-a_0Q_i.
\end{equation}

Simultaneous similarity and a change of cyclic vector preserve this relation
space \cite[Theorem 2.2]{LiYuanOrbits}.

Put
\begin{equation}\label{eq:intro-universal}
 R_r=\mathbb Z[A_0,\ldots,A_{r-1}],\qquad
 G=x^r+A_{r-1}x^{r-1}+\cdots+A_0.
\end{equation}
The symbols $q_{b,i}$ and $Q_i$ also denote the universal remainder coefficients and
columns.  For $I=\{i_1<\cdots<i_s\}\subseteq\{0,\ldots,n-1\}$, with $1\le s\le r$, set
\begin{equation}\label{eq:intro-block}
 H_I=(Q_{i_1}\ \cdots\ Q_{i_s}),\qquad
 J_I^{(s)}=I_s(H_I),\qquad X_I^{(s)}=V(J_I^{(s)})\subseteq\A_{\mathbb Z}^r.
\end{equation}
We use $I_j(M)$ for the ideal of $j$-minors of a matrix $M$, and put
$R_{r,K}=R_r\otimes_{\mathbb Z}K$ and $X_{I,K}^{(s)}=X_I^{(s)}\times_{\mathbb Z}K$.
For support calculations write
\begin{equation}\label{eq:intro-AUB}
 A=I\cap\{0,\ldots,r-1\},\qquad
 U=I\cap\{r,\ldots,n-1\},\qquad
 B=\{0,\ldots,r-1\}\setminus A.
\end{equation}
The high-column count is $u_I=|U|$; the notation $\ell(\boldsymbol m)$ is
reserved for necklace occupancy.  In Section~\ref{sec:monodromy},
$B^{\rm prim}$ denotes a primitive exponent support rather than the missing-low-row set $B$.

Appendix~\ref{v6:app:coordinates} gives alternative root-jet and Schur
coordinates; reconstruction does not depend on them.

Monic division identifies the kernel of the remainder map with
multiples of $G$.  The resulting exact sequence and complementary-minor
identities are recorded in Theorem~\ref{thm:Toeplitz-duality}.
The next section uses the corresponding multiplication presentation
for an arbitrary binary linear system.

\section{Integral reconstruction from the first equation bundle}\label{sec:reconstruction-new}
\label{sec:toeplitz-secant}\label{v6:lin:section}

For fixed $1\le s\le r\le N$, every rank-$s$ linear system of binary
forms has the same determinantal postulation.  Its first equations
nevertheless recover the linear system in families.  We establish the
common geometry directly from multiplication, then prove this recovery.

\subsection{Binary multiplication and the universal family}

\begin{lemma}[Universal binary multiplication]\label{lem:binary-multiplication}
For integers $0\le r\le N$, binary multiplication
\[
 \mu:\PP^r_{\mathbb Z}\times_{\mathbb Z}\PP^{N-r}_{\mathbb Z}
 \longrightarrow\PP^N_{\mathbb Z},\qquad([G],[H])\longmapsto[GH]
\]
is finite, faithfully flat and Gorenstein of rank $\binom Nr$.  Its relative
dualizing sheaf is
\begin{equation}\label{eq:binary-multiplication-dualizing}
 \omega_\mu\simeq
 \omega_{(\PP^r\times\PP^{N-r})/\mathbb Z}
 \otimes\mu^*\omega_{\PP^N/\mathbb Z}^{-1}
 \simeq\mathcal O(N-r,r).
\end{equation}
These assertions, including the displayed line bundle, commute with arbitrary base change.
\end{lemma}

\begin{proof}
Over an algebraically closed field, a binary form has finitely many
degree-$r$ factor divisors, obtained by distributing its root
multiplicities, including infinity; at least one distribution exists.
Thus the projective morphism $\mu$ is quasi-finite and surjective,
hence finite and surjective.  Source and target are regular and smooth
over $\mathbb Z$ of relative dimension $N$.  For $z\mapsto y$, the
finite residue-field extension and the dimension formula give
$\dim\mathcal O_z=\dim\mathcal O_y$.  Miracle flatness therefore
gives flatness \cite[Tag~00R4]{Stacks}.

Above a squarefree form there are $\binom Nr$ factorizations.  They are
reduced: for coprime $G,H$, the inverse image of $K\cdot GH$ under
$(\dot G,\dot H)\mapsto H\dot G+G\dot H$ is spanned by $(G,0)$
and $(0,H)$, so the projective tangent map is an isomorphism.
Connectedness of the target now gives the stated finite-flat rank.

The graph of $\mu$ is a regular immersion, as a section of a smooth
projection.  Hence $\mu$ is a local-complete-intersection morphism
of virtual relative dimension zero.  Its dualizing sheaf is the
determinant of its cotangent complex.  Using
$\mu^*\mathcal O(1)=\mathcal O(1,1)$ gives
\eqref{eq:binary-multiplication-dualizing}, whose invertibility proves
the Gorenstein assertion.  For a finite projective algebra $B/R$,
\[
 \operatorname{Hom}_R(B,R)\otimes_R R'
 \simeq\operatorname{Hom}_{R'}(B\otimes_R R',R')
\]
compatibly with the algebra action.  This proves arbitrary base-change
compatibility of the dualizing module and its line-bundle description;
finite local freeness preserves the other assertions.  The endpoint
maps $r=0,N$ are identities.
\end{proof}

Put
\[
 S_n=H^0(\PP^1_{\mathbb Z},\mathcal O(n)),\qquad
 E=S_N^*,\qquad L=N-r+1,\qquad h=r-s+1.
\]
All bundle constructions below are over $\mathbb Z$ and commute with
base change.
Under $\operatorname{Sym}^r(\PP^1)\simeq\PP^r$, a binary form $G$ represents an effective
divisor $D_G$ of degree $r$.  Let $\cE_{N,r}$ be the rank-$r$ tautological secant bundle with
fiber
\begin{equation*}
 (\cE_{N,r})_{D}=H^0(D,\mathcal O_D(N)).
\end{equation*}
The universal divisor projects finitely and flatly to $\PP^r$: it is
$\PP^1\times\PP^{r-1}$ with the degree-$r$ multiplication map of
Lemma~\ref{lem:binary-multiplication}.  Its twisted ideal sequence remains
exact under base change.  Pushing forward, using
$H^1(\PP^1,\mathcal O(N-r))=0$, gives
\begin{equation}\label{eq:secant-bundle-resolution}
 0\longrightarrow
 H^0(\PP^1,\mathcal O(N-r))\otimes\mathcal O_{\PP^r}(-1)
 \longrightarrow
 H^0(\PP^1,\mathcal O(N))\otimes\mathcal O_{\PP^r}
 \longrightarrow\cE_{N,r}\longrightarrow0.
\end{equation}
This is the classical Schwarzenberger presentation of the secant
bundle \cite{Schwarzenberger}. For $N>r$, see also
\cite[Section~2.2, equation~(2.5a), arXiv version]{RaicuSam}.
Secant schemes of a linear series in this form,
including their diagonal local models, are classical \cite{HuibregtseJohnsen}.

For a scheme $T$ and a rank-$s$ subbundle
$V\subset S_N\otimes\mathcal O_T$ with locally free quotient, define
\begin{equation}\label{v6:lin:degeneracy}
 X_r(V)=D_{s-1}\bigl(V\otimes\mathcal O_{\PP(S_r)_T}
                    \longrightarrow\cE_{N,r,T}\bigr)
 \subset\PP(S_r)_T.
\end{equation}
Here $\cE_{N,r}$ is the secant bundle in
\eqref{eq:secant-bundle-resolution}, and projective space parametrizes
lines of binary forms.  No base-point-free hypothesis is imposed on $V$.
Projection after multiplication by the universal form gives
\begin{equation}\label{v6:lin:projected-multiplication}
 T_V(G):S_{N-r}\otimes\mathcal O(-1)
       \longrightarrow (S_N\otimes\mathcal O_T/V)\otimes\mathcal O.
\end{equation}
This matrix has $L+h-1$ rows and $L$ columns.  In general its rows are
linear combinations of the rows of the full multiplication matrix;
row deletion is the special case of a monomial linear system.

\begin{theorem}[Uniform family of linear-system secant schemes]
\label{v6:lin:uniform}
The maximal minors of \eqref{v6:lin:projected-multiplication} define
$X_r(V)$.  Over every field, their homogeneous ideal is saturated and
has height $h$.  Its quotient is Cohen--Macaulay of dimension $s$ and
has the resolution, Betti numbers and Hilbert series in
\eqref{eq:projective-EN}--\eqref{eq:projective-Hilbert-vector}.

Over $\operatorname{Gr}_{\mathbb Z}(s,S_N)$, the universal construction
is flat, and its homogeneous Eagon--Northcott resolution remains exact
after every base change.  In particular it defines a Hilbert morphism
\begin{equation}\label{v6:lin:hilbert-map}
 \Theta:\operatorname{Gr}_{\mathbb Z}(s,S_N)
 \longrightarrow\operatorname{Hilb}^{P}(\PP(S_r)/\mathbb Z)
\end{equation}
for the common Hilbert polynomial $P$.
\end{theorem}

\begin{proof}
Quotienting the middle term of \eqref{eq:secant-bundle-resolution} by
$V$ identifies the cokernels in \eqref{v6:lin:degeneracy} and
\eqref{v6:lin:projected-multiplication}; their $(r-s)$th Fitting ideals coincide.
Over an algebraically closed field the locus is the image of
$\{([G],[H]):[GH]\in\PP(V)\}$.
By Lemma~\ref{lem:binary-multiplication} this incidence is nonempty
and finite faithfully flat over $\PP(V)$, of dimension $s-1$.
Thus $\dim X_r(V)\le s-1$; the determinantal height bound gives the
opposite inequality, and hence height $h$.
The expected grade makes the homogeneous Eagon--Northcott complex
exact \cite[Chapter~2]{BrunsVetter}; see also \cite{EagonNorthcott}.
Its shifts are $L+i-1$ and its ranks are
$\binom{L+h-1}{L+i-1}\binom{L+i-2}{i-1}$, giving the asserted
resolution and Hilbert series.  The standard-graded calculation is
recorded in Appendix~\ref{r18:sec:standard-postulation}.
Since the quotient is Cohen--Macaulay of dimension $s\ge1$,
its zeroth local cohomology at the irrelevant ideal vanishes,
and the ideal is saturated.  These assertions
descend to every field.

For the universal family use divided powers
$\Gamma^j(M^*)=(\operatorname{Sym}^j M)^*$ in Eagon--Northcott,
as in \eqref{eq:EN-terms}.  Fix a polynomial degree.  The resulting
Eagon--Northcott complex is finite locally free, and its residue-field
complexes are exact in positive homological degrees.  Over a local
base ring the leftmost nonzero differential has a unit maximal column
minor, so it is a split injection.  The next differential vanishes on
its image.  Cancel this contractible summand and repeat
\cite[Tag~00MT]{Stacks}.  Only a free module in homological degree zero
remains.  These finite operations work over any local ring and, in
each fixed polynomial degree, extend to a neighbourhood.  Thus every
graded quotient is locally free and the complex stays exact after
every tensor product; no common splitting for all degrees is required.
Homogeneous localization and its degree-zero summand preserve
base-flatness, so Proj gives \eqref{v6:lin:hilbert-map}.
Since both the ambient structure sheaf and $\mathcal O_{X_r(V)}$ are
base-flat, the ideal sequence also remains exact after any base change.
It therefore recovers the actual pulled-back ideal sheaf.  Pullback
from the Grassmannian treats every $T$, including nonreduced and
non-Noetherian bases, without imposing regularity on $T$.
\end{proof}

\begin{remark}\label{v6:lin:chart-scope}
An arbitrary $V$ can have fixed roots, including a fixed root at
infinity.  A selected monic chart need not be dense.  For monomial systems
containing $X^N$, Theorem~\ref{thm:secant-compactification} proves
density separately by following every incidence component.
\end{remark}

\subsection{An integral isomorphism for maximal minors}

Use the standard monomial bases and write
$G=\sum_{k=0}^r A_kX^kZ^{r-k}$.  The full multiplication matrix is
\[
 M(G)=(A_{i-j})_{\substack{0\le i\le N\\0\le j<L}},
 \qquad A_k=0\quad\text{for }k\notin[0,r].
\]
Taking its maximal minors defines
\begin{equation}\label{v6:lin:psi}
 \Psi:\bigwedge^L E\longrightarrow
 H^0(\PP(S_r),\mathcal O(L)).
\end{equation}
Before choosing a basis, the natural map $\Psi_{\mathrm{nat}}$ has source
$\bigwedge^L E\otimes\det S_{N-r}$.  The increasing monomial basis
trivializes this determinant factor in \eqref{v6:lin:psi}.
With $(m,n)=(L,N+1)$, its dual is the multiplication-induced Hermite map of
\cite[Section~2.4, Theorem~2.9, arXiv version]{RaicuSam}, after matching
determinant generators: their decreasing generator differs from ours
by $(-1)^{\binom L2}$.  It sends a divided power of $G$ to the wedge of its $L$ monomial
multiples.  The proof fixes the integral normalization.

\begin{proposition}[Integral maximal-minor isomorphism]
\label{v6:lin:minor-isomorphism}
The map $\Psi$ is an isomorphism of free $\mathbb Z$-modules.
Consequently it is an isomorphism after base change to every ring.
\end{proposition}

\begin{proof}
For a row subset $0\le i_0<\cdots<i_{L-1}\le N$, the diagonal
term of the associated minor is
\begin{equation}\label{v6:lin:diagonal}
 \prod_{j=0}^{L-1}A_{i_j-j}.
\end{equation}
The sequence $(i_j-j)$ is weakly increasing and belongs to $[0,r]$:
the upper bound follows from
$i_j\le N-(L-1-j)=r+j$.
Conversely, each weakly increasing sequence of $L$ integers in $[0,r]$
defines one row subset.  The diagonal terms thus biject with all
degree-$L$ monomials in $A_0,\ldots,A_r$.

Give $A_k$ weight $-k^2$.  A nonzero term of the minor corresponding
to a column permutation $\sigma$ has weight
$-\sum_j(i_j-\sigma(j))^2$.  The terms
$\sum_j i_j^2$ and $\sum_j\sigma(j)^2$ do not depend on $\sigma$,
and the strict rearrangement inequality uniquely maximizes
$\sum_j i_j\sigma(j)$ at the identity permutation.  Thus
\eqref{v6:lin:diagonal} is the unique term of greatest weight, and its
coefficient is $+1$.  Order the monomial basis by weight, breaking
ties arbitrarily, and order the row subsets by their diagonal
monomials.  The resulting coefficient matrix is unitriangular.
Its determinant is $1$, which proves the integral assertion.
\end{proof}

\subsection{Recovery in families}

For a subbundle $V$ as above, let
$W=V^\perp\subset E\otimes\mathcal O_T$.  Its rank is
$N+1-s=L+h-1$.

\begin{theorem}[Reconstruction and Hilbert closed immersion]
\label{v6:lin:reconstruction}
Let $\rho:\PP(S_r)_T\to T$ be the projection.  There are no
nonzero homogeneous equations of $X_r(V)$ of degree less than $L$,
and its first equation bundle is
\begin{equation}\label{v6:lin:first-equations}
 \rho_*\mathcal I_{X_r(V)}(L)
   =\Psi\bigl(\bigwedge^L W\bigr)
   \subset H^0(\PP(S_r),\mathcal O(L))\otimes\mathcal O_T.
\end{equation}
This equality commutes with arbitrary base change.  The embedded
scheme $X_r(V)$ recovers $V$ functorially, and the Hilbert morphism
$\Theta$ in \eqref{v6:lin:hilbert-map} is a closed immersion.
\end{theorem}

\begin{proof}
\emph{The first-equation subbundle.} Cauchy--Binet identifies the maximal minors of the projected
multiplication matrix with the images of $\bigwedge^L W$ under $\Psi$
over every ring.  Thus this identification specifies the map into
$H^0(\mathcal O(L))$, not just an abstract isomorphism of bundles.

Put $D=S_{N-r}\otimes\mathcal O_T$.  By Theorem~\ref{v6:lin:uniform},
the sheafified ideal resolution remains exact over $T$.  Placed in
cohomological degrees $-i$, its terms are
\[
 C^{-i}=\bigwedge^{L+i}W\otimes\Gamma^i D\otimes\det D
                \otimes\mathcal O(-L-i),\qquad 0\le i\le h-1.
\]
The fixed monomial basis trivializes $\det D$ as in
\eqref{v6:lin:psi}.  After twisting by $L$, the terms with $i>0$ are
acyclic under $\rho$, since $1\le i\le h-1=r-s\le r-1$.
Projective-space cohomology and its natural base-change maps hold over
every ring \cite[Tags~01XT, 01XV, 01XW]{Stacks}.
The bounded \v Cech double complex on the standard projective cover has
\[
 E_1^{-i,q}(n)=R^q\rho_*C^{-i}(n)
 \;\Longrightarrow\;R^{q-i}\rho_*\mathcal I_{X_r(V)}(n).
\]
At $n=L$, only $E_1^{0,0}$ survives.  Hence the derived direct image is
$\bigwedge^L W\otimes\det D$ in degree zero, and its map into
$\rho_*\mathcal O(L)$ is the Cauchy--Binet map.
Locally splitting $E\otimes\mathcal O_T=W\oplus W'$ makes
$\bigwedge^L W$ a direct summand of $\bigwedge^L(E\otimes\mathcal O_T)$;
Proposition~\ref{v6:lin:minor-isomorphism} transports this summand to
the first equations.  Thus \eqref{v6:lin:first-equations} is an equality
of subbundles with its natural base-change map, not merely a fiberwise
rank calculation.

\emph{No lower-degree equations.} For $0\le n<L$, every term $C^{-i}(n)$ has negative twist.  Its only
possible nonzero relative cohomology is in degree $r$, and contributes
in total degree $r-i\ge r-(h-1)=s\ge1$, never in degree zero.
Thus $\rho_*\mathcal I_{X_r(V)}(n)=0$, even when the higher cohomology
of some negative twists is nonzero.  For $n<0$, the same vanishing
follows from $\mathcal I(n)\subset\mathcal O(n)$.  In particular no
additional low-degree equations are hidden on a nonreduced base.

\emph{Recovery by contraction.} Applying $\Psi_{\mathrm{nat}}^{-1}$ to the
first equation bundle and then tensoring by $(\det D)^{-1}$ recovers
$U=\bigwedge^L W$ independently of a determinant trivialization.
In the fixed bases this is simply $\Psi^{-1}$.  Recover $W$ by contraction:
\begin{equation}\label{v6:lin:contraction}
 W=\operatorname{im}\left(
 U\otimes\bigwedge^{L-1}E^*
 \longrightarrow E\otimes\mathcal O_T\right).
\end{equation}
Every such contraction belongs to $W$.  Conversely, locally extend
a basis of the direct summand $W$ to a basis of $E\otimes\mathcal O_T$.
Any prescribed basis vector of $W$ occurs in an $L$-fold wedge of
basis vectors of $W$, because $\operatorname{rank}W\ge L$.
Contracting against the duals of the other $L-1$ factors recovers
that vector, up to a unit sign.  This proves
\eqref{v6:lin:contraction} over arbitrary rings, without division by
$L$ or a factorial.  Choosing one such preimage for each basis vector
gives a local right inverse.  Together with the locally split inclusion
$W\subset E\otimes\mathcal O_T$, this proves that the contraction
image commutes with arbitrary base change; this is false for arbitrary
module-map images.  Finally the locally split evaluation sequence
\[
 0\longrightarrow V\longrightarrow S_N\otimes\mathcal O_T
 \longrightarrow W^*\longrightarrow0
\]
recovers $V$ as the annihilator, compatibly with base change.
For $L=1$ the contraction is the inclusion of $U=W$, so this endpoint
requires no separate division argument.

\emph{The Hilbert morphism.} For every scheme $T$, two $T$-points defining the same embedded
Hilbert family therefore have the same $U$, $W$, and $V$.
These are equalities on $T$, not just on its geometric fibers.
Thus $\Theta$ is a monomorphism of functors on all test schemes.
Its graph is closed because the Hilbert scheme is
separated over $\mathbb Z$; the projection from the product is proper
because the Grassmannian is projective.  Thus $\Theta$ is proper.
A proper monomorphism is a closed immersion
\cite[Tag~04XV]{Stacks}.
\end{proof}

\begin{remark}[What is reconstructed]\label{v6:lin:meaning}
The theorem concerns the entire embedded determinantal scheme in the
fixed coefficient space $\PP(S_r)$, with $N,r,s$ fixed.  It does not
assert recovery from the abstract scheme, its reduced support, one
component, or its numerical invariants.  For a monomial $V_I$, the
fixed monomial coordinates recover $I$ from $V_I$.  If $r=s$, then $\rank W=L$, and the first equation is the
Pl\"ucker line transported by $\Psi$.  With parameters $k=r-1$ and
$n=N$, this is precisely the Pl\"ucker--Hermite/Schwarzenberger
construction of the Poncelet hypersurface in
\cite[Section~3.2]{IlardiSupinoValles}.  Hermite reciprocity in arbitrary
characteristic is treated in \cite{RaicuSam}; as explained before
Proposition~\ref{v6:lin:minor-isomorphism}, its multiplication-induced
map, with those determinant generators identified, is the dual of $\Psi$.
For general $s\le r$, the reconstruction statement concerns the complete
first-equation bundle and the whole embedded determinantal family after
arbitrary base change, together with the contraction back to the defining
linear system.
The determinant-line convention
in \eqref{v6:lin:psi} must be retained when discussing equivariance.
Theorems on powers of binary forms and their derived reciprocity
\cite{RaicuSamWeymanYang} concern a neighboring ideal problem, not the
whole family of divisor schemes used here.
\end{remark}

\begin{example}[Uniform postulation, different geometry]\label{ex:early-cubic}
Over a field of characteristic zero, take $N=4$ and $r=s=2$, with
$G=A_2X^2+A_1XY+A_0Y^2$ in $\PP^2$ and monic chart $A_2=1$.
The two supports give
\[
 I=\{0,4\}:\quad A_1(2A_0A_2-A_1^2)=0,
 \qquad I'=\{3,4\}:\quad A_0^3=0.
\]
Indeed, modulo $x^2+A_1x+A_0$, the coefficient of $x$ in $x^4$ is
$A_1(2A_0-A_1^2)$, whereas $\det(Q_3,Q_4)=A_0^3$.
Homogenization gives the displayed cubics. The first is a line and a smooth
conic; the second is a triple line. Both have resolution $0\to S(-3)\to S\to S/(f)\to0$,
Hilbert series $(1+t+t^2)/(1-t)^2$, degree $3$, regularity $2$,
and generic initial ideal $(x_0^3)$. On the monic chart, their root weights are respectively
$3$ and $6$. Uniform projective data therefore coexist with distinct
component structures, multiplicities, and weighted geometry.
\end{example}

\section{Relation incidence and arithmetic profiles}\label{sec:projective-incidence}
\label{sec:monodromy}

The relation projection is finite flat; the divisor projection can
contract components.  Arithmetic profiles distinguish their reduced
degrees, image dimensions, and scheme multiplicities.

\subsection{The finite relation projection}
\label{r19:sec:relation-projection}
Over a field $K$, retain $1\le s\le r\le N$, $|I|=s$, and $\max I=N$.
Let $\mu:\PP^r\times\PP^{N-r}\to\PP^N$ be binary-form multiplication and define
\begin{equation}\label{eq:projective-factorization-incidence}
 \overline{\mathfrak Z}_I
 =\mu^{-1}(\PP(V_I))
 =\{([G],[H]):GH\in\PP(V_I)\}.
\end{equation}
The two projections organizing the rest of the paper are
\begin{equation}\label{eq:incidence-two-projections}
 \overline X_I^{(s)}\xleftarrow{\ p\ }
 \overline{\mathfrak Z}_I\xrightarrow{\ q\ }\PP(V_I),\qquad
 p([G],[H])=[G],\quad q([G],[H])=[GH].
\end{equation}
Write $H_G,H_H$ for the factor hyperplane classes, and use throughout
$\eta=p^*c_1(\mathcal O_{\PP^r}(1))=H_G$ and
$\xi=q^*c_1(\mathcal O_{\PP(V_I)}(1))=H_G+H_H$,
also for their restrictions and pullbacks.

Write $Z_{I,K}=p^{-1}(\{A_r\ne0\})$ for the selected-monic chart
of $\overline{\mathfrak Z}_I$; its relation coordinate remains projective,
as in \eqref{eq:incidence}.  The common monic relation chart is
$\cF_{I,r}=\Fact_r(F_t)$, where $F_t=x^N+\sum_{i\in I\setminus\{N\}}t_ix^i$.
The relevant charts are collected in Table~\ref{tab:incidence-dictionary}.

\begin{table}[H]
\centering
\caption{Relation projections and image projections, with their chart conventions.}
\label{tab:incidence-dictionary}
\begin{tabular}{@{}p{.19\textwidth}p{.34\textwidth}p{.40\textwidth}@{}}
\toprule
Space & Projection & Property\\
\midrule
$\overline{\mathfrak Z}_I$ & $q:\overline{\mathfrak Z}_I\to\PP(V_I)$
 & Finite flat Gorenstein of rank $\binom Nr$.\\
$\overline{\mathfrak Z}_I$ & $p:\overline{\mathfrak Z}_I\to\overline X_I^{(s)}$
 & Projective; may contract profile components.\\
$Z_{I,K}$ & $\varpi:Z_{I,K}\to\PP(V_I)$
 & Monic selected divisor; quasi-finite, generally not proper.\\
$Z_{I,K}$ & $p:Z_{I,K}\to X_{I,K}^{(s)}$
 & Projective; fibers $\PP(\ker H_I(g))$.\\
$\cF_{I,r}$ & $\cF_{I,r}\to\A_K^{s-1}$
 & Monic relation; finite flat of rank $\binom Nr$.\\
\bottomrule
\end{tabular}
\end{table}

\begin{theorem}[Projective factorization complete intersection]
\label{thm:projective-factorization-CI}
The scheme $\overline{\mathfrak Z}_I$ is a complete intersection of $N-s+1$ divisors of
bidegree $(1,1)$ and has dimension $s-1$.  The map
$q:\overline{\mathfrak Z}_I\to\PP(V_I)$ is finite flat Gorenstein of degree
$\binom Nr$, and the projection to $\PP^r$ has scheme-theoretic image
$\overline X_I^{(s)}$ and is an isomorphism over an open subset
containing every generic point of the image.
In $A_{s-1}(\PP^r\times\PP^{N-r})$, its fundamental cycle is
\begin{equation}\label{eq:incidence-class}
 [\overline{\mathfrak Z}_I]=(H_G+H_H)^{N-s+1}.
\end{equation}
For $a+b=s-1$ its mixed degrees are
\begin{equation}\label{eq:mixed-degrees}
 \int_{\overline{\mathfrak Z}_I}H_G^aH_H^b
 =\binom{N-s+1}{r-a}.
\end{equation}
Moreover,
\begin{equation}\label{eq:incidence-canonical-class}
 \omega_{\overline{\mathfrak Z}_I}
 \simeq\mathcal O_{\overline{\mathfrak Z}_I}(N-r-s,r-s).
\end{equation}
\end{theorem}

\begin{proof}
The $N-s+1$ linear equations of $\PP(V_I)\subset\PP^N$ pull back under
$\mu$ to the coefficients outside $I$ of $GH$.  By
Lemma~\ref{lem:binary-multiplication}, this pullback is flat and $q$
is finite flat Gorenstein of rank $\binom Nr$.  The bilinear equations
therefore form a regular sequence, giving purity of dimension $s-1$
and \eqref{eq:incidence-class}.

The set-theoretic image under $p$ is the degeneracy locus
\eqref{eq:projective-secant-locus}.  Over a top-dimensional image
component the general relation fiber has dimension zero; otherwise
that image would have dimension at most $s-2$.  Hence the general
evaluation rank is $s-1$, the relation is unique, and so is $H=F/G$.
The projective-kernel equations make $p$ factor scheme-theoretically
through $\overline X_I^{(s)}$.  Wherever an $(s-1)$-minor is invertible,
eliminating the projective relation coordinates identifies $p$ with an
isomorphism, including nilpotents; for $s=1$ the $0$-minor is the unit.
These opens contain every generic point of $\overline X_I^{(s)}$.
By Theorem~\ref{v6:lin:uniform}, this scheme is Cohen--Macaulay with no
embedded associated points.  The ideal
$\ker(\mathcal O_{\overline X_I^{(s)}}\to
p_*\mathcal O_{\overline{\mathfrak Z}_I})$ vanishes at all generic
points and is therefore zero, proving the scheme-theoretic image assertion.

Extracting the coefficient of $H_G^rH_H^{N-r}$ gives
\eqref{eq:mixed-degrees}.  Finally
$\omega_q=\mathcal O(N-r,r)$ and
$q^*\omega_{\PP(V_I)}=\mathcal O(-s,-s)$ give
\eqref{eq:incidence-canonical-class}.
\end{proof}

\begin{remark}[The gcd restriction test]\label{prop:factorization-secant-smoothness}
At $z=([G],[H])\in\overline{\mathfrak Z}_I(K)$, put $Q=\gcd(G,H)$,
$d=\deg Q$, and write $D_Q$ for its divisor.  Set
\begin{equation}\label{eq:factorization-secant-restriction}
 \operatorname{rk}_I(Q)=\rank\bigl(V_I\longrightarrow
 H^0(D_Q,\mathcal O_{D_Q}(N))\bigr).
\end{equation}
With both quantities zero for $Q=1$, one has
\begin{equation}\label{eq:factorization-incidence-tangent-dimension}
 \dim_KT_z\overline{\mathfrak Z}_I=s-1+d-\operatorname{rk}_I(Q).
\end{equation}
Indeed, for $S_j=H^0(\PP^1,\mathcal O(j))$, the multiplication differential
has image $HS_r+GS_{N-r}=(Q)_N$: its summands intersect in
$(GH/Q)S_d$, so their sum has dimension $N-d+1$.
Its cokernel modulo $V_I$ has dimension $d-\operatorname{rk}_I(Q)$.
Subtracting the resulting Jacobian rank from the ambient dimension $N$
proves the formula.  The complete intersection is therefore smooth at $z$
exactly when the restriction is surjective; on a monic chart this is
\begin{equation}\label{eq:factorization-secant-rank}
 \rank(x^i\bmod Q)_{i\in I}=d.
\end{equation}
This is a test at a factorization point, not a criterion for smoothness of
its coefficient-space image.
\end{remark}

The resultant $\Res(G,H)$ is a section of
$\mathcal O(N-r,r)$ on the product.  Comparing \eqref{eq:incidence-canonical-class} with
$q^*\omega_{\PP(V_I)}\simeq\mathcal O(-s,-s)$ gives the line-bundle identity
\begin{equation}\label{eq:projective-RH-line-bundle}
 \omega_{\overline{\mathfrak Z}_I}
 \simeq q^*\omega_{\PP(V_I)}\otimes\mathcal O(N-r,r).
\end{equation}
The corresponding K\"ahler-different calculation is not needed for the reconstruction--profile arguments below and is not pursued here.

\subsection{The image projection and its generic relation fiber}
\label{r19:sec:image-projection}
Evaluation rank controls contraction under $p$ over every field.

\begin{lemma}[Rank and image of an incidence component]
\label{lem:profile-rank-image}
Let $K$ be a field, let $1\le s\le r\le N$ with $|I|=s$ and $\max I=N$,
and let $W$ be a reduced irreducible
component of $\overline{\mathfrak Z}_I$.  Put $Y=p(W)$ with its reduced
structure, write $\zeta_Y$ for its generic point, and let $k$ be the generic rank on $Y$ of the evaluation map
$V_I\otimes\mathcal O\to\cE_{N,r}$ (equivalently, of $H_I$ on the monic
chart).  Then
\begin{equation}\label{eq:profile-rank-image}
 \dim Y=k,\qquad
 W_{\zeta_Y}=\PP(\ker H_I(\zeta_Y))
 \simeq\PP^{s-k-1}_{K(Y)}.
\end{equation}
At nonmonic divisors, $H_I$ denotes secant evaluation.  On a nonempty
open $Y^\circ\subseteq Y$, with kernel bundle
$\mathcal E=\ker(V_I\otimes\mathcal O_{Y^\circ}\to
\cE_{N,r}|_{Y^\circ})$, one has
\[
 W|_{Y^\circ}=\PP_{Y^\circ}(\mathcal E),\qquad
 K(W)=K(Y)(u_1,\ldots,u_{s-k-1}).
\]
Distinct reduced incidence components have distinct images.
\end{lemma}

\begin{proof}
By Theorem~\ref{thm:projective-factorization-CI}, the incidence is pure of
dimension $s-1$.  On the integral base $Y$, all $(k+1)$-minors of
the evaluation map vanish identically.  Inverting a nonzero $k$-minor
(the unit $0$-minor if $k=0$) gives a nonempty open $Y^\circ$ with split
normal form $\operatorname{diag}(I_k,0)$.  Thus the full relation
scheme there is the integral projective bundle
$\PP_{Y^\circ}(\mathcal E)$, of dimension $\dim Y+s-k-1$.
The universal-divisor sequence \eqref{eq:secant-bundle-resolution}
identifies this bundle functorially with the factorization incidence
by the unique quotient $H=F/G$.  As a locally closed subscheme of
the pure incidence, it gives $\dim Y\le k$.

Conversely, the generic fiber of the integral $(s-1)$-dimensional
$W$ is an integral closed subscheme of $\PP^{s-k-1}_{K(Y)}$.
The dimension formula gives $\dim Y\ge k$.  Equality follows,
and this full-dimensional integral subscheme is the whole projective
space, scheme-theoretically.
The defining ideal of $W|_{Y^\circ}$ in this bundle vanishes after
localization to $K(Y)$.  The bundle is flat over the integral base, so its
structure sheaf is torsion-free over that base; hence this ideal is zero.
A trivialization of $\mathcal E$ gives the function-field assertion.  If two components had the same image,
both generic fibers would be the full relation fiber; their closures
would coincide, proving the last assertion.
\end{proof}

\subsection{The monic relation chart and its boundary}

Let $K$ be algebraically closed of characteristic zero and fix
\begin{equation}\label{eq:general-support-normalization}
 I=\{a=i_0<i_1<\cdots<i_{s-1}=N\},\qquad 2\le s\le r<N.
\end{equation}
Set
\begin{equation}\label{eq:lattice-index-data}
 M=N-a,\qquad d=\gcd\{i-a:i\in I\},\qquad m=M/d,
\end{equation}
and write $I=a+dB^{\rm prim}$, where $0,m\in B^{\rm prim}$ and $\gcd(B^{\rm prim})=1$.  Over the sparse coefficient space
$\A_K^{s-1}$ consider
\begin{equation}\label{eq:general-sparse-polynomial}
 F_t(x)=x^N+\sum_{i\in I\setminus\{N\}}t_ix^i
       =x^a\overline F_t(x^d).
\end{equation}
For the canonical support, $t_i\mapsto-t_i$ matches the earlier
minus-sign convention.  The boundary $c_N=0$ is taken in $Z_{I,K}$,
where the selected divisor remains monic.

\begin{proposition}[Incidence chart]\label{prop:incidence-factorization-chart}
Assume that $K$ is algebraically closed of characteristic zero, with support as in
\eqref{eq:general-support-normalization}.
The morphism $\cF_{I,r}\to\A_K^{s-1}$ is finite flat of degree $\binom Nr$.  It is naturally
the chart $c_N\ne0$ in the incidence scheme $Z_{I,K}$ of \eqref{eq:incidence}.  The complementary
incidence locus $c_N=0$ maps into the block for $I\setminus\{N\}$ and hence has image dimension
at most $s-2$.
\end{proposition}

\begin{proof}
On $c_N=1$, the equation is $g\mid F_t$, with a unique monic quotient
of degree $N-r$.  This identifies the chart; its finite flat degree is
the restriction of the cover in Theorem~\ref{thm:projective-factorization-CI}.
On $c_N=0$ the relation is supported in $I\setminus\{N\}$, so
Theorem~\ref{thm:expected-codimension} gives the image bound $s-2$.
\end{proof}

Put
\begin{equation}\label{eq:regular-profile-base}
 U_I^{\rm reg}=\{t_a\ne0,\ \overline F_t\text{ has }m\text{ distinct nonzero roots}\}
 \subseteq\A_K^{s-1}.
\end{equation}

\begin{lemma}[Boundary exclusion]\label{lem:profile-boundary-exclusion}
Assume that $K$ is algebraically closed of characteristic zero, with support as in
\eqref{eq:general-support-normalization}.
Let $C$ be a reduced irreducible component of the finite flat factorization chart
$\cF_{I,r}$.  Then $C$ dominates $\A_K^{s-1}$ and
\begin{equation}\label{eq:profile-boundary-image-bound}
 \dim\operatorname{im}\bigl(C\setminus C|_{U_I^{\rm reg}}\longrightarrow X_I^{(s)}\bigr)
 \le s-2.
\end{equation}
The image in $X_I^{(s)}$ of the complementary
relation-boundary locus $c_N=0$ also has dimension at most $s-2$.
\end{lemma}

\begin{proof}
A nonzero function on the integral coefficient base is a nonzerodivisor
in its finite flat factorization algebra, so no minimal prime lies
over a nonzero base prime.  Every $C$ therefore dominates the base.
The open set $U_I^{\rm reg}$ is nonempty: taking $t_a=1$
and all other nonleading coefficients equal to zero gives
$\overline F_t(y)=y^m+1$, which has distinct nonzero roots
in characteristic zero.
The complement of $U_I^{\rm reg}$ is the proper closed set
$V(t_a\Disc(\overline F_t))$.  Its inverse image in the finite cover $C$
has dimension at most $s-2$, as does its image after forgetting the
relation.  The $c_N=0$ assertion is
Proposition~\ref{prop:incidence-factorization-chart}.
\end{proof}

\subsection{Monodromy, profiles, and occupancy}

Over $U_I^{\rm reg}$ the root $0$ has fixed multiplicity $a$, while the
$M=dm$ nonzero roots form $m$ blocks of $d$ roots.  Their geometric monodromy is
\begin{equation}\label{eq:sparse-root-monodromy}
 \Gamma_I=C_d\wr\mathfrak S_m=C_d^m\rtimes\mathfrak S_m
\end{equation}
as stated in Esterov--Lang \cite[Proposition~1.5]{EsterovLang}
(Theorem~1.5 of arXiv:1812.07912v2).
Esterov--Lang note immediately after that proposition that the
univariate result was proved earlier by Dvornicich and Zannier.
The support after removing $x^a$ and dividing by $d$ is primitive,
the endpoint coefficients are nonzero, and $x\mapsto x^d$ is
\emph{\'etale} on $\Gm$.  The unnormalized family on this chart
is the product of $\Gm$ with the monic family, so leading-coefficient
normalization preserves the root cover.

\begin{lemma}[Characteristic-zero transport]\label{lem:monodromy-transport}
The orbit decomposition supplied by \eqref{eq:sparse-root-monodromy} is valid after base
change to every algebraically closed field of characteristic zero.
\end{lemma}

\begin{proof}
Work first over $\overline{\mathbb Q}$, choosing the roots of unity
there.  The nonzero-root, ordered-nonzero-root and subset covers
associated with $\overline F_t(x^d)$ are finite
\emph{\'etale} over the regular coefficient open.  Their connected
components are unchanged by any algebraically closed extension
\cite[Tag 0363]{Stacks}.  After extension to $\mathbb C$,
Esterov--Lang give the stated geometric monodromy.  Choosing an embedding
$\overline{\mathbb Q}\hookrightarrow K$ transports the same components
and labeled subset orbits to $K$; no embedding of $K$ into $\mathbb C$
is required.
\end{proof}

\begin{lemma}[Uniform subset rank]\label{lem:uniform-subset-rank}
Assume that $K$ is algebraically closed of characteristic zero.
Let $y_1,\ldots,y_m$ be the generic roots of the primitive sparse polynomial
$\overline F_t(y)=\sum_{b\in B^{\rm prim}}t_by^b$, with the leading coefficient normalized to one.
For every $k\le s-1$ and every $k$-subset $J\subseteq\{1,\ldots,m\}$,
\begin{equation}\label{eq:uniform-subset-rank}
 \rank (y_j^b)_{j\in J,\ b\in B^{\rm prim}}=k.
\end{equation}
\end{lemma}

\begin{proof}
A polynomial of degree at most $m$ vanishing on all the generic roots
is proportional to $\overline F_t$, so the full matrix has rank $s-1$.
For $k\le s-1$, a nonzero $k$-minor exists.  The $\mathfrak S_m$
monodromy is transitive on $k$-subsets and fixes the coefficient
columns, so Galois conjugation supplies one on every prescribed row subset.
On the finite \emph{\'etale} ordered-root cover, the finitely many
rank-failure loci are closed and their images miss the generic point.
Their complement gives a common nonempty open in $U_I^{\rm reg}$.
\end{proof}

Choose an integer $e$ in the range
\begin{equation}\label{eq:e-range-all}
 \max\{0,r-M\}\le e\le\min\{a,r\},\qquad q_e=r-e.
\end{equation}
A binary necklace $\nu$ of length $d$ is a $C_d$-orbit of a binary word,
with Hamming weight $w(\nu)$, orbit size $o(\nu)$, and rotation-stabilizer
order $h_\nu=d/o(\nu)$.  A necklace profile consists of nonnegative
multiplicities $(m_\nu)_\nu$ satisfying
\begin{equation}\label{eq:profile-constraints}
 \sum_\nu m_\nu=m,\qquad
 \sum_\nu m_\nu w(\nu)=q_e.
\end{equation}
It describes an orbit of $\Gamma_I$ on $q_e$-subsets of the nonzero roots.  Define its occupancy
\begin{equation}\label{eq:profile-occupancy}
 \ell(\boldsymbol m)=\sum_{w(\nu)>0}m_\nu.
\end{equation}

\paragraph{Admissible factorization profiles.}
A pair $\pi=(e,\boldsymbol m)$ satisfying
\eqref{eq:e-range-all} and \eqref{eq:profile-constraints} is called an
\emph{admissible profile}.  Let $\cF_\pi$ be its open-and-closed
factorization piece over $U_I^{\rm reg}$, and let
$E_\pi=(\cF_\pi)_{\rm red}$.  Write $W_\pi$ for its reduced closure
in $\overline{\mathfrak Z}_I$, and
$\overline Y_\pi=p(W_\pi)$ with the reduced structure.  Table~\ref{tab:profile-data} summarizes these data.

\begin{table}[H]
\centering
\caption{How profile data enter the relation cover, image, and full cycle.}
\label{tab:profile-data}
\begin{tabular}{@{}p{.24\textwidth}p{.69\textwidth}@{}}
\toprule
Data & Geometric information\\
\midrule
Necklace multiset & Reduced branch $W_\pi$ and degree of $q|_{W_\pi}$.\\
Zero-root split $e$ & Artin factor of length $\binom ae$ in the full incidence.\\
Occupancy $\ell$ & Generic evaluation rank; image dimension and fiber of $p$.\\
Fixed-$e$ sum & Reduced-profile mixed degrees; no factor $\binom ae$.\\
$\binom ae$-weighted total & Mixed degrees of the full, possibly nonreduced incidence.\\
\bottomrule
\end{tabular}
\end{table}

\begin{theorem}[Dominant-profile component theorem]
\label{thm:dominant-profile-components}
Assume that $K$ is algebraically closed of characteristic zero, with support as in
\eqref{eq:general-support-normalization}.  The reduced irreducible components of
$\overline{\mathfrak Z}_I$ are the closures $W_\pi$ indexed by admissible profiles
$\pi=(e,\boldsymbol m)$.  Over $U_I^{\rm reg}$, their reduced degrees above
relation space are
\begin{equation}\label{eq:profile-orbit-degree}
 D(\boldsymbol m)=\frac{m!}{\prod_\nu m_\nu!}\prod_\nu o(\nu)^{m_\nu},
\end{equation}
and their scheme-theoretic generic multiplicities in the full incidence are
$\binom ae$.  Put $\overline Y_\pi=p(W_\pi)$, with its reduced structure.
Then
\begin{equation}\label{eq:profile-rank}
 k_\pi=\rank H_I(\zeta_{\overline Y_\pi})
      =\min\{\ell(\boldsymbol m),s-1\},
 \qquad \dim\overline Y_\pi=k_\pi.
\end{equation}
Distinct profiles have distinct projective images, and the generic fiber of
$W_\pi\to\overline Y_\pi$ is $\PP^{s-k_\pi-1}_{K(\overline Y_\pi)}$.
Consequently,
\begin{equation}\label{eq:X-component-classification}
 \begin{aligned}
 \Irr\bigl((\overline X_I^{(s)})_{\mathrm{red}}\bigr)
 &\simeq\Irr\bigl((X_I^{(s)})_{\mathrm{red}}\bigr)\\
 &\simeq\{\pi=(e,\boldsymbol m):\ell(\boldsymbol m)\ge s-1\}.
 \end{aligned}
\end{equation}
For each such dominant profile, $p:W_\pi\to\overline Y_\pi$ is birational,
and the ambient determinantal schemes $\overline X_I^{(s)}$ and
$X_I^{(s)}$ have generic multiplicity $\binom ae$ along
$\overline Y_\pi$ and its monic chart, respectively.  Profiles of smaller occupancy
are contracted to dimension $\ell(\boldsymbol m)$.
\end{theorem}

\begin{proof}
\emph{Relation branches and their multiplicities.}
Over $U_I^{\rm reg}$, coprimality of $x^a$ and the squarefree
$\overline F_t(x^d)$ gives open-and-closed pieces
\begin{equation}\label{eq:factorization-e-splitting}
 \Fact_r(x^a\overline F_t(x^d))
 \simeq\coprod_e\Fact_e(x^a)\times\Fact_{q_e}(\overline F_t(x^d)).
\end{equation}
The algebra $\mathcal A_{a,e}=\mathcal O(\Fact_e(x^a/K))$ is supported
at $(x^e,x^{a-e})$.  Over $x^a$ both factors are monic, so it is the
entire fiber of the multiplication map in Lemma~\ref{lem:binary-multiplication}
and has length $\binom ae$,
including length one at $e=0,a$; compare the general
factorization-algebra basis theorem in \cite{LaksovThorup}.  The second factor is finite
\emph{\'etale} over the regular base.  Its connected components are
integral and correspond to wreath-product orbits: assigning types to blocks and then
rotating their words gives \eqref{eq:profile-orbit-degree}.
The generic local algebra
$\mathcal A_{a,e}\otimes_K K(W_\pi)$ has residue-field length
$\binom ae$ and $K(\PP(V_I))$-dimension $\binom ae D(\boldsymbol m)$.

Finite flatness, as in Lemma~\ref{lem:profile-boundary-exclusion},
makes every projective incidence component dominate relation space.
Thus the closures $W_\pi$ exhaust all components.

\emph{Images and contraction.}
At a generic profile point all monomials in $I$ vanish modulo $x^e$,
since $a\ge e$.  At a selected nonzero root $\alpha$, dividing the
evaluation row by $\alpha^a$ gives
$((\alpha^d)^b)_{b\in B^{\rm prim}}$: one row per occupied block.
Lemma~\ref{lem:uniform-subset-rank} gives rank
$\min\{\ell(\boldsymbol m),s-1\}$: when at least $s-1$ blocks are occupied,
choose $s-1$ of them, and use the sparse relation for the matching upper bound.
This proves the rank formula.  Lemma~\ref{lem:profile-rank-image} now gives
the image dimensions, the full generic projective fibers, and distinctness
of the images.  Since $p$ is proper and covers the pure $(s-1)$-dimensional
scheme $\overline X_I^{(s)}$ set-theoretically, its maximal-dimensional images
are precisely the irreducible components of that scheme.  Every $W_\pi$ meets
the chart with nonzero leading relation coefficient, where $G$ is monic;
thus every image meets the monic chart, giving \eqref{eq:X-component-classification}.

\emph{Transferring the scheme multiplicity.}
Invert a nonzero $(s-1)$-minor at a generic dominant image point.
Row and column operations put the restriction matrix in the form
$\left(\begin{smallmatrix}I_{s-1}&v\\0&b\end{smallmatrix}\right)$.
The determinantal ideal is generated by the entries of $b$.  The projective
kernel equations force the final relation coordinate to be nonzero, and
normalizing it to one eliminates the other coordinates as $-v$; the remaining
equations are exactly $b=0$.  Thus the full incidence and the determinantal
image are scheme-theoretically isomorphic on this open set, including
nilpotents.  The previously computed generic incidence length is consequently
the generic multiplicity of the ambient determinantal scheme along its reduced image component.  The rank--image lemma also makes the
map of reduced dominant components birational.
\end{proof}

Proposition~\ref{prop:formal-profile-model} gives the corresponding
completed local models.  Cohen--Macaulayness now turns generic
multiplicities into a reducedness criterion.

\begin{corollary}[Geometric reducedness]\label{cor:geometric-reducedness}
Let $K$ be a field of characteristic zero, and let $I$ satisfy
\eqref{eq:general-support-normalization}, with $a=\min I$.
Then
\[
 X_I^{(s)}\text{ is geometrically reduced}
 \quad\Longleftrightarrow\quad
 \overline X_I^{(s)}\text{ is geometrically reduced}
 \quad\Longleftrightarrow\quad a\le1.
\]
If $a=1$, both schemes are geometrically reduced but not geometrically
integral.  For primitive supports $d=1$, they are geometrically integral
when $a=0$, have exactly two geometric irreducible components when $a=1$,
and are nonreduced when $a\ge2$.
\end{corollary}

\begin{proof}
Over an algebraic closure, Theorems~\ref{thm:expected-codimension},
\ref{thm:secant-compactification}, and
\ref{thm:dominant-profile-components} make both schemes
Cohen--Macaulay and pure of dimension $s-1$, with minimal components
meeting the regular dominant-profile open.  If $a\le1$, their generic
lengths are $\binom ae=1$; thus $R_0$ and $S_1$ imply reducedness.

Suppose that $a\ge2$, and put $e_0=\max\{1,r-M\}$.  Since the support
contains $s$ distinct integers between $a$ and $N$, one has $M\ge s-1$.
The inequalities $r<N=a+M$ and $r\ge s$ give
\[
 1\le e_0\le a-1,\qquad e_0\le r-s+1,\qquad
 s-1\le q_{e_0}=r-e_0\le M.
\]
Also $m\ge s-1$, since $B^{\rm prim}$ has $s$ distinct elements
in $[0,m]$.  Choose one nonzero root in each
of $s-1$ distinct blocks and then any further $q_{e_0}-(s-1)$ roots.
This is possible since $q_{e_0}\le M$, and it produces a dominant
profile with $0<e_0<a$.  The ambient generic multiplicity along its
image is $\binom a{e_0}>1$.  Both the projective stratum and its monic
chart are consequently nonreduced.

When $a=1$, both $e=0,1$ are admissible and have $q_e\ge s-1$.
The same construction gives dominant profiles with distinct images,
hence at least two components.  If $d=1$, the nonzero-root monodromy is
$\mathfrak S_M$, transitive on subsets of each fixed cardinality.
There is exactly one profile for each admissible $e$, proving the two
primitive classifications.  Nonreducedness for $a\ge2$ also descends
to $K$: in characteristic zero a reduced finite-type $K$-scheme is
geometrically reduced.
\end{proof}

\section{Normal relation covers and product normalizations}\label{sec:normalization}
\label{v6:norm:section}

Over the algebraically closed characteristic-zero field of
Section~\ref{sec:monodromy}, write $\widetilde W_\pi$ and
$\widetilde Y_\pi$ for the normalizations of the reduced branch and image.
For arbitrary sparse support we classify the former over fixed
relation space; complete progressions also permit image-normality tests.

Write $n_\nu=m_\nu$ and $w_\nu=w(\nu)$, retaining $h_\nu$;
then \eqref{eq:profile-constraints} is
$\sum_\nu n_\nu=m$, $\sum_\nu n_\nu w_\nu=r-e$.

\subsection{The normal cover over an arbitrary sparse relation space}

Identify $R_I=\PP(K^I)$ with the relation space $\PP(V_I)$ in the
fixed monomial basis, and let $L/K(R_I)$
be the splitting field of the generic relation after its fixed factor
$X^a$ is removed.  On its regular open set the Galois group is
$\Gamma=C_d^m\rtimes\mathfrak S_m$, by
\eqref{eq:sparse-root-monodromy}.  Let $\mathcal T_I$ be the
normalization of $R_I$ in $L$.

\begin{proposition}[Classification of normal relation covers]
\label{v6:norm:cover}
For every admissible profile, choose rotation coordinates in the $m$
root blocks and put
\begin{equation}\label{v6:norm:stabilizer}
 H_\pi=\prod_{\nu:n_\nu>0}
       \bigl(C_{h_\nu}^{\,n_\nu}\rtimes\mathfrak S_{n_\nu}\bigr)
       \ \subseteq\ \Gamma.
\end{equation}
There is an isomorphism over $R_I$
\begin{equation}\label{v6:norm:cover-quotient}
 \widetilde W_\pi\simeq\mathcal T_I/H_\pi.
\end{equation}
For two profiles on this fixed support, their normal relation covers
are isomorphic over $R_I$ if and only if
\begin{equation}\label{v6:norm:cover-classification}
 \{\!\{(h_\nu,n_\nu):n_\nu>0\}\!\}_\pi
 =
 \{\!\{(h_\nu,n_\nu):n_\nu>0\}\!\}_{\pi'}.
\end{equation}
The multiplicities in these multisets are retained.
\end{proposition}

\begin{proof}
On the regular relation locus a profile is the transitive
$\Gamma$-set of selected root subsets.  Block rotations preserving the selection form $C_{h_\nu}$, and
permutations preserve necklace types.  Choosing representatives gives
the stabilizer \eqref{v6:norm:stabilizer}.  The finite branch over $R_I$ has function field $L^{H_\pi}$, so
$\widetilde W_\pi$ is the integral closure in that field.  On an affine base open $\operatorname{Spec}A_0$, let
$C$ be the integral closure of $A_0$ in $L$; it is finite over $A_0$
since $A_0$ is a finite-type algebra over a field
\cite[Tag 0335]{Stacks}.  Then $C^{H_\pi}=C\cap L^{H_\pi}$ is exactly
the integral closure of $A_0$ in $L^{H_\pi}$.  These invariant rings
localize on base opens and give the finite quotient
$\mathcal T_I/H_\pi$, proving \eqref{v6:norm:cover-quotient}.

An isomorphism of these normal covers is equivalent to an isomorphism
of the intermediate fields over $K(R_I)$, or to conjugacy of the
corresponding subgroups in $\Gamma$.  Indeed, an isomorphism between
intermediate fields of a finite Galois extension extends to an
automorphism of that extension \cite[Tag 0BME]{Stacks}; conversely,
such an automorphism preserves integrality over each base open and
extends across the boundary.  The image of $H_\pi$ in
$\mathfrak S_m$ has the block sets of cardinalities $n_\nu$ as its
orbits, including singleton orbits.  Its intersection with the base
group $C_d^m$ restricts on every coordinate of such an orbit to the
unique subgroup of order $h_\nu$.  Conjugation in the wreath product
permutes these orbits and preserves these orders.  Thus conjugacy
implies \eqref{v6:norm:cover-classification}.  Conversely, equality
of the multisets permits a permutation matching the corresponding
block sets; it conjugates the standard subgroups
\eqref{v6:norm:stabilizer}.  Empty and full necklaces both have
$h_\nu=d$, but remain separate Young blocks when both occur;
repeated pairs and singleton blocks are therefore not amalgamated.
The fixed allocation $e$ enters neither subgroup.  This proves the converse.
\end{proof}

\begin{remark}\label{v6:norm:cover-scope}
These are covers of the fixed relation space; weight, necklace
shape, and allocation $e$ can be forgotten.  For general sparse
support this does not identify the image normalization.

At a prime divisor of $R_I$, the orbit lengths of geometric inertia
on $\Gamma/H_\pi$ give the normal cover's ramification indices,
as follows over the strict henselization of the corresponding valuation
ring in characteristic zero.  Image conductors require separate data.
\end{remark}

\begin{remark}\label{v6:norm:ordered-root-warning}
The normalization in the definition of $\mathcal T_I$ is essential.
On the monic chart a natural ordered-root model has equations
\[
 e_j(u_1^d,\ldots,u_m^d)=0
 \qquad\text{for }m-j\notin B^{\rm prim}.
\]
These equations need not define a normal scheme;
Proposition~\ref{v6:norm:cover} uses its integral closure.
\end{remark}

\subsection{Both normalizations for complete progressions}

We now impose the stronger hypothesis
\begin{equation}\label{v6:norm:full-progression}
 I=a+d\{0,1,\ldots,m\},\qquad s=m+1.
\end{equation}
Fix a primitive $d$th root of unity $\zeta$.  For each necklace choose
a subset $S_\nu\subseteq\mathbb Z/d\mathbb Z$ representing it, so
$|S_\nu|=w_\nu$.  The classes $\eta,\xi$ retain their meanings from
Section~\ref{sec:projective-incidence}.

\begin{theorem}[Product normalizations and their polarizations]
\label{v6:norm:product}
Under \eqref{v6:norm:full-progression}, every admissible profile has
\begin{equation}\label{v6:norm:product-formulas}
 \widetilde W_\pi\simeq
       \prod_{\nu:n_\nu>0}\PP^{n_\nu},\qquad
 \widetilde Y_\pi\simeq
       \prod_{\substack{\nu:n_\nu>0\\w_\nu>0}}\PP^{n_\nu}.
\end{equation}
The lift of $W_\pi\longrightarrow\overline Y_\pi$ to these
normalizations forgets the factor belonging to the empty necklace.
If $H_\nu$ is the hyperplane class of the corresponding factor, then
\begin{equation}\label{v6:norm:polarizations}
 \eta=\sum_\nu\frac{w_\nu}{h_\nu}H_\nu,
 \qquad
 \xi=\sum_\nu\frac d{h_\nu}H_\nu.
\end{equation}
The second formula is on $\widetilde W_\pi$; the empty-necklace
coefficient in the first formula is zero.
\end{theorem}

\begin{proof}
\emph{Relation normalization.}
Parameterize an ordered root block by $[u:v]\in\PP^1$.  Its relation
factor and its selected factor are, respectively,
\begin{equation}\label{v6:norm:single-block}
 v^dX^d-u^dZ^d,
 \qquad
 G_\nu(u,v;X,Z)=\prod_{b\in S_\nu}(vX-\zeta^b uZ).
\end{equation}
Their products, with fixed factors $X^a,X^e$ and complementary
selected factors from $S_\nu^c$, define an $H_\pi$-invariant
morphism $(\PP^1)^m\to\overline{\mathfrak Z}_I$, including
both endpoints.  Quotienting rotations and then within-type
permutations gives the smooth product
\[
 (\PP^1)^m/H_\pi
 \simeq\prod_\nu
       \operatorname{Sym}^{n_\nu}(\PP^1/C_{h_\nu})
 \simeq\prod_\nu\PP^{n_\nu}.
\]
The power maps and the map forgetting the partition into types are
finite, so the map to relation space, and hence to $W_\pi$, is finite.
On the regular locus the quotient identifies exactly this profile's
subsets.  It is therefore finite birational from a normal variety,
proving the first formula in \eqref{v6:norm:product-formulas}.

\emph{Image normalization.}
Empty-block parameters disappear from the selected factor.  For every
nonempty type the map
$\PP^1/C_{h_\nu}\to\PP(S_{w_\nu})$ in
\eqref{v6:norm:single-block} is nonconstant, hence finite onto its image.
On the ordered quotient-root product, iterating the finite
multiplication of Lemma~\ref{lem:binary-multiplication} gives a finite
map to the image.  Passing to the within-type symmetric quotients
retains properness and finite fibers, hence finiteness.
For a general selected factor, partition its nonzero roots by
equality of their $d$th powers.  Each part recovers its necklace
type and, after fixing $S_\nu$, its parameter modulo $C_{h_\nu}$.
Generic injectivity in characteristic zero gives birationality,
proving the second formula
and identifying the lifted projection.  If no block is occupied,
the empty product is a point, with image $[X^e]$.

\emph{Polarizations.}
The subset $S_\nu$ is a union of $C_{h_\nu}$-orbits, so
$h_\nu\mid w_\nu$.  Put $q=d/h_\nu$ and write
$S_\nu=T_\nu+q\{0,\ldots,h_\nu-1\}$ with
$T_\nu\subseteq\mathbb Z/q\mathbb Z$.  Grouping each orbit in
\eqref{v6:norm:single-block} gives the factors
$v^{h_\nu}X^{h_\nu}-\zeta^{t h_\nu}u^{h_\nu}Z^{h_\nu}$,
$t\in T_\nu$.  Their product is a nonzero form at every quotient parameter
$[u^{h_\nu}:v^{h_\nu}]$, including both endpoints, and has parameter
degree $w_\nu/h_\nu$; the relation factor has degree $d/h_\nu$.
The reduced mask $T_\nu$ is aperiodic by maximality of the stabilizer,
and the Fourier coefficients in \eqref{v6:norm:fourier} satisfy
$\lambda_{\nu,h_\nu j}=h_\nu\sum_{t\in T_\nu}(\zeta^{h_\nu})^{tj}$.  Pulling back along
$(\PP^1)^{n_\nu}\longrightarrow\operatorname{Sym}^{n_\nu}\PP^1$
identifies its hyperplane bundle with the exterior product of the
$n_\nu$ copies of $\mathcal O_{\PP^1}(1)$.  The multiplication
formulas therefore give \eqref{v6:norm:polarizations}.  Multiplying
every selected factor by $X^e$ is a linear embedding of factor
spaces and does not change these line bundles.
\end{proof}

\begin{corollary}[What the polarized normalization remembers]
\label{v6:norm:polarized-classification}
For complete progressions, the abstract image normalization is
classified by the unordered list of positive integers $n_\nu$ for
occupied types.  Its isomorphism class with $\eta$ is classified by
the unordered list
\[
 \bigl(n_\nu,w_\nu/h_\nu\bigr).
\]
For dominant profiles, the isomorphism class of the normalization
with the two divisor classes $\eta,\xi$ is classified by
\[
 \bigl(n_\nu,w_\nu/h_\nu,d/h_\nu\bigr).
\]
This last assertion concerns two line bundles, not the two
morphisms defined by their specified linear systems.
\end{corollary}

\begin{proof}
By \eqref{eq:profile-rank}, dominance means all blocks are occupied,
so the two products coincide.  Their primitive nef rays are the factor
hyperplane classes, whose morphisms recover the factor dimensions.
Isomorphisms therefore permute equal-dimensional factors, and every
such permutation is realized.  The coefficients in
\eqref{v6:norm:polarizations} give the three classifications.
\end{proof}

\section{Polarizations and normality of profile images}\label{sec:normality-new}

Normality requires monic parameter recovery and a separate check of
allocations at infinity.

Assume for the moment that the profile in
\eqref{v6:norm:full-progression} is dominant, so $w_\nu>0$ for every
type present.  Remove the fixed factor $X^e$ from the selected form
and put $R=r-e$.  Define the Fourier coefficients
\begin{equation}\label{v6:norm:fourier}
 \lambda_{\nu,k}=\sum_{b\in S_\nu}\zeta^{bk},
 \qquad
 \mathcal D_\nu=\{h_\nu,2h_\nu,\ldots,n_\nu h_\nu\}.
\end{equation}
The coefficient $\lambda_{\nu,k}$ vanishes automatically when
$h_\nu\nmid k$.  Let $Y_\pi^{\rm mon}$ be the affine chart in
$\overline Y_\pi$ on which the remaining selected form is monic.

\begin{theorem}[Monic normality criterion]
\label{v6:norm:affine-criterion}
The affine variety $Y_\pi^{\rm mon}$ is normal if and only if both
of the following conditions hold:
\begin{enumerate}[label=\textup{(\Alph*)},leftmargin=*]
\item the sets $\mathcal D_\nu$ are pairwise disjoint;
\item $\lambda_{\nu,h_\nu j}\ne0$ for all $\nu$ and
      $1\le j\le n_\nu$.
\end{enumerate}
When they hold, $Y_\pi^{\rm mon}\simeq\A^m$, and in particular it
is smooth.
\end{theorem}

\begin{proof}
\emph{The actual coefficient algebra.}
The inverse image of this chart in
Theorem~\ref{v6:norm:product} is $\prod_\nu\A^{n_\nu}$.  Write
$\beta_{\nu,i}=u_{\nu,i}^{h_\nu}$ on its ordered monic root cover,
and let $e_{\nu,j}$ and $P_{\nu,j}$ be the elementary symmetric
functions and power sums in the $\beta_{\nu,i}$.  Its coordinate
ring is
\begin{equation}\label{v6:norm:affine-normalization-ring}
 \mathcal B=K[e_{\nu,j}:1\le j\le n_\nu],
 \qquad \deg e_{\nu,j}=h_\nu j.
\end{equation}
The $k$th power sum of the roots of the selected factor is
\begin{equation}\label{v6:norm:output-moments}
 Q_k=\sum_{\nu:h_\nu\mid k}
       \lambda_{\nu,k}P_{\nu,k/h_\nu}.
\end{equation}
Newton's identities in characteristic zero identify the coordinate
ring of its image with
\begin{equation}\label{v6:norm:affine-image-ring}
 \mathcal A=K[Q_1,\ldots,Q_R]\subseteq\mathcal B.
\end{equation}
All higher power sums belong to this ring as well.  By
Theorem~\ref{v6:norm:product}, $\mathcal B$ is the finite
birational normalization of $\mathcal A$.

\emph{Necessity.}
Consider the indecomposables
$\mathcal B_+/(\mathcal B_+)^2$ for the positive grading in
\eqref{v6:norm:affine-normalization-ring}.  Newton's identities give
\[
 P_{\nu,j}\equiv(-1)^{j-1}j e_{\nu,j}
       \pmod{(\mathcal B_+)^2}\quad(j\le n_\nu),
\]
whereas $P_{\nu,j}$ has zero class for $j>n_\nu$.
Products of positive-degree outputs have zero class in this quotient.
At weight $k$, only $Q_k$ can contribute.  It cannot span two
indecomposables of that weight or recover one with zero Fourier
coefficient.  Thus $\mathcal A=\mathcal B$ requires (A) and (B);
normality implies this equality since $\mathcal B$ is its normalization.

\emph{Sufficiency.}
Suppose (A) and (B) hold.  At weight $k=h_\nu j$,
$j\le n_\nu$, formula \eqref{v6:norm:output-moments} has precisely
one indecomposable term, with nonzero coefficient
$(-1)^{j-1}j\lambda_{\nu,k}e_{\nu,j}$.  Its other terms are
polynomials in generators of smaller positive weight.  After these
have been recovered, subtract their polynomial from $Q_k$ and divide
by the nonzero scalar $(-1)^{j-1}j\lambda_{\nu,k}$; this recovers
$e_{\nu,j}$ inside the actual algebra $\mathcal A$.  The induction
uses no parameter denominators and works on the whole chart,
including repeated roots and the total-zero point.
All the needed indices occur in \eqref{v6:norm:affine-image-ring}:
indeed $h_\nu j\le h_\nu n_\nu\le w_\nu n_\nu\le R$.
Hence $\mathcal A=\mathcal B$, proving both the criterion and the
asserted affine-space isomorphism.
\end{proof}

\begin{theorem}[Normality of the projective image]
\label{v6:norm:projective-criterion}
The entire projective image $\overline Y_\pi$ is normal if and only
if (A) and (B) of Theorem~\ref{v6:norm:affine-criterion} hold and,
in addition,
\begin{enumerate}[label=\textup{(\Alph*)},start=3,leftmargin=*]
\item the bounded weighted-sum map
\begin{equation}\label{v6:norm:endpoint-injectivity}
 \prod_\nu\{0,1,\ldots,n_\nu\}\longrightarrow\mathbb Z,
 \qquad (k_\nu)_\nu\longmapsto\sum_\nu w_\nu k_\nu
\end{equation}
is injective.
\end{enumerate}
When these conditions hold, the normalization morphism is an
isomorphism
\[
 \overline Y_\pi\simeq\prod_\nu\PP^{n_\nu}.
\]
Thus normality and smoothness are equivalent in this family.  Here
normality means normality of the variety, not projective normality
of its homogeneous coordinate ring.
\end{theorem}

\begin{proof}
\emph{Necessity and endpoint allocations.}
Conditions (A) and (B) are necessary by restriction to the monic
chart.  If (C) fails, take two different vectors $(k_\nu)_\nu$ with
the same weighted sum.  In the normalization put $k_\nu$ parameters
of type $\nu$ at infinity and all its remaining parameters at zero.
They define distinct normalization points with the same binary
monomial output, whose infinity multiplicity is
$\sum_\nu w_\nu k_\nu$.
A quotient-root parameter at infinity contributes $w_\nu$ output roots.
The finite birational normalization morphism is therefore not an
isomorphism, so its target is not normal.

\emph{Uniqueness on geometric points.}
Suppose that all three conditions hold.  For any selected
form, the multiplicity of its root at infinity and (C) recover
uniquely the number $k_\nu$ of parameters of each type at infinity.
After removing those infinite roots, its finite part comes from the
truncated profile with multiplicities $n_\nu-k_\nu$.  Conditions
(A) and (B) are inherited by every such truncation.  The same
polynomial reconstruction applies, even if the truncated parameters
fall outside the original Krylov inequalities; a zero multiplicity
simply removes a type, and degree zero gives $K$.
It therefore recovers the finite parameters uniquely.  The
normalization morphism is consequently injective on geometric
points.

\emph{Scheme structure at the boundary.}
Write $f:\mathcal N\to\overline Y_\pi$ for the finite normalization.
At a closed point $y$, put $A_y=\mathcal O_{\overline Y_\pi,y}$ and
$B_y=(f_*\mathcal O_{\mathcal N})_y$.  Exact completion gives
$\widehat A_y\hookrightarrow B_y\otimes_{A_y}\widehat A_y
\simeq\prod_{x\in f^{-1}(y)}\widehat{\mathcal O}_{\mathcal N,x}$
\cite[Tags 00MA and 07N9]{Stacks}.  Only the already proved uniqueness
reduces this whole-fiber product to a single factor.

The finite and infinite factors of the output at $y$ are coprime.
After fixing scalar normalizations, multiplication is \emph{\'etale}
there by Lemma~\ref{lem:binary-multiplication}.
Hensel factorization in the completed ambient coefficient ring
therefore supplies both factor coefficient germs in $\widehat A_y$.
Equivalently, the formal inverse of coprime multiplication is
defined on the ambient coefficient completion and then restricts
to the quotient $\widehat A_y$.

For the finite factor, the reconstruction by weighted Newton
identities in Theorem~\ref{v6:norm:affine-criterion} recovers all
the local normalization parameters.  At infinity use the reciprocal
coordinate and the reciprocal necklaces.  Their relevant Fourier
coefficients are $\lambda_{\nu,-h_\nu j}$, which are nonzero if
and only if $\lambda_{\nu,h_\nu j}$ is nonzero: the cyclotomic
automorphism of $\mathbb Q(\zeta)$ sending $\zeta$ to $\zeta^{-1}$
interchanges them; no embedding of $K$ into $\mathbb C$ is required.
The same affine reconstruction applies to the multiplicities
$k_\nu$.  The symmetric parameters of the finite and infinite
clusters are local coordinates on the source product, since their
corresponding divisors on each parameter $\PP^1$ have disjoint
support.  Repeated roots within either factor require no further
splitting.  The reconstruction therefore places in the image of
$\widehat A_y$ a full regular system of parameters for the unique smooth
source factor $\widehat{\mathcal O}_{\mathcal N,x}$.  Since
$B_y\otimes_{A_y}\widehat A_y$ is finite over the complete noetherian local
ring $\widehat A_y$, the quotient by the image of $\widehat A_y$ is a
finite $\widehat A_y$-module and the image is $\mathfrak m_y$-adically
closed.  It contains the polynomial algebra in the recovered parameters and
hence, by completeness, their full power-series algebra.  Thus
\[
 \widehat A_y \xrightarrow{\;\sim\;}
 B_y\otimes_{A_y}\widehat A_y
 \simeq \widehat{\mathcal O}_{\mathcal N,x}.
\]
Now $B_y/A_y$ is a finite $A_y$-module.  Completion is faithfully flat over
the noetherian local ring $A_y$, so the displayed isomorphism gives
$(B_y/A_y)\otimes_{A_y}\widehat A_y=0$ and hence $B_y=A_y$
\cite[Tag 00MC]{Stacks}.  This holds at every closed point, so the coherent
normalization quotient vanishes globally.
\end{proof}

\begin{remark}\label{v6:norm:empty-and-marked}
For a nondominant profile, delete the empty-necklace factor before
applying the affine and projective criteria to its image, and replace
$m$ by its occupancy.  The same proofs apply.  The relation
normalization still has the additional factor of
Theorem~\ref{v6:norm:product}.

Even for a dominant profile, the normal relation morphism
$\widetilde W_\pi\longrightarrow R_I$ need not descend through
$\widetilde Y_\pi\longrightarrow\overline Y_\pi$.  The two
normalizations coincide as varieties in that case, but a boundary
identification in the latter map can join points with different
relations.  Accordingly, a classification retaining the normal
relation morphism is a classification of that marked diagram.
\end{remark}

\subsection{Prime-step consequence}
For a necklace $S\subseteq C_d$, write $M_S(T)=\sum_{a\in S}T^a$.
\begin{corollary}[Prime step]
\label{v6:norm:prime-step}
If $d$ is prime, every homogeneous profile of a nonempty proper
necklace has smooth projective image for every admissible
multiplicity.  More generally a dominant profile has normal image
exactly in the following cases: it contains at most one proper
necklace type, and, if a full necklace occurs together with that
type, the multiplicity of the proper type is less than $d$.
The full-necklace-only case is included whenever admissible.
\end{corollary}

\begin{proof}
For a nonempty proper subset $S$, the polynomial $M_S$ cannot be
divisible by $\Phi_d=1+T+\cdots+T^{d-1}$: its degree is at most
$d-1$, and its coefficients are zero or one with at least one of
each.  Thus all its Fourier coefficients are nonzero.  Every
proper nonempty necklace is aperiodic.  Two distinct proper types
would both have $1$ in their degree sets, violating (A).  For a
single proper type of multiplicity $n$, its degree set is
$\{1,\ldots,n\}$; the full necklace has $h=d$ and degree set
$\{d,2d,\ldots,n_{\rm full}d\}$.  These are disjoint exactly when
$n<d$, provided the full type is present.  Condition (B) holds for
both types.  Finally, if $w$ is the proper weight, then
$1\le w<d$ and $\gcd(w,d)=1$.  Equality
$wk+dl=wk'+dl'$ with $0\le k,k'\le n<d$ forces $k=k'$ and
$l=l'$.  Thus (C) holds as well.  The one-type cases have (C)
automatically, and the result follows from the projective
criterion.
\end{proof}

\section{Concluding remarks and further directions}\label{sec:conclusion-new}
The results above separate two levels of information.  The present paper
determines the embedded divisor scheme, its arithmetic profile branches, and
their normalizations and natural polarizations.  A finer problem is to
determine which functions on a normalization come from the actual profile
image, or equivalently how much of the nonnormal image algebra is remembered
by the normalized geometry.  This problem requires additional information
and lies beyond the scope of the present paper.

The appendices collect the technical material supporting the arguments
above: root and Schur coordinates, fixed-support flatness and postulation,
profile degrees and formal local models, and the canonical factorization and
zero-fiber specializations used in those models.  Each statement keeps its
stated hypotheses; in particular, the profile and monodromy assertions are in
characteristic zero unless explicitly stated otherwise.

Natural continuations include determining which coarser embedded data, short
of the complete divisor scheme, still recover the defining linear system;
understanding arithmetic profiles beyond the characteristic-zero monodromy
regime; and finding intrinsic criteria that recognize the product
normalizations without first choosing a complete progression.
\appendix
\section{Root and Schur coordinates; weighted resolutions}\label{v6:app:coordinates}

The hook coordinates below support Appendix~\ref{app:auxiliary-inputs};
their root-weighted resolutions complement the standard-graded
resolution of Corollary~\ref{cor:projective-resolution}.
Proposition~\ref{v6:lin:minor-isomorphism} is independent of this appendix.

\subsection{Root jets}

Suppose that over a splitting field $L/K$,
\begin{equation}\label{eq:root-factorization}
 g(x)=\prod_{\nu=1}^{m}(x-\lambda_\nu)^{e_\nu},
 \qquad \lambda_\mu\ne\lambda_\nu\ (\mu\ne\nu),
 \qquad \sum_\nu e_\nu=r.
\end{equation}
For the Hasse derivative $D^{(j)}$, one has
$D^{(j)}x^i=\binom{i}{j}x^{i-j}$ for $i\ge j$, and zero for $i<j$.
The root-jet matrix, with the same zero convention, is
\begin{equation}\label{eq:root-jet}
 H_g^{\rm jet}=
 \left(\binom{i}{j}\lambda_\nu^{i-j}\right)_{
 (\nu,j),\ 0\le j<e_\nu;\ 0\le i<n}.
\end{equation}
Hermite--Hasse evaluation in degrees less than $r$ identifies
$H_g^{\rm jet}$ and $H_g^{\rm rem}$ by an invertible row transformation over $L$
\cite[Theorem 3.4]{LiYuanOrbits}.  Ordinary derivatives would lose
multiplicity information when factorials vanish.

\subsection{Rectangular Schur coordinates}

Let $X_1,\ldots,X_r$ be variables, let $s_\kappa(X)$ be the Schur polynomial of a
partition $\kappa$, and use Vieta's isomorphism
\begin{equation}\label{eq:vieta}
 R_r\xrightarrow{\ \sim\ }\mathbb Z[X_1,\ldots,X_r]^{\mathfrak S_r},
 \qquad A_{r-j}\longmapsto(-1)^je_j(X).
\end{equation}
Write $S_\kappa\in R_r$ for the inverse image of $s_\kappa$.  If $g$ has root multiset
$\Lambda_g$, then $S_\kappa(g)=s_\kappa(\Lambda_g)$ in every splitting field.

For $U=\{u_1<\cdots<u_{r-1}\}$ with $u_1\ge1$, put
\begin{equation}\label{eq:kappa-U}
 \kappa(U)=
 \bigl(u_{r-1}-(r-1),u_{r-2}-(r-2),\ldots,u_1-1,0\bigr).
\end{equation}
For an $r$-subset $L=\{\ell_0<\cdots<\ell_{r-1}\}$, set
$U_L=\{\ell_1-\ell_0,\ldots,\ell_{r-1}-\ell_0\}$.  The rectangular
Schur--Pl\"ucker theorem of \cite[Theorem 3.3]{LiYuanSchur} states
\begin{equation}\label{eq:prior-schur-plucker}
 \det(Q_{\ell_0},\ldots,Q_{\ell_{r-1}})
 =\bigl((-1)^rA_0\bigr)^{\ell_0}S_{\kappa(U_L)}.
\end{equation}
For $L\subseteq\{0,\ldots,n-1\}$, the reduced partitions are exactly those in the rectangle
$(n-r)^{r-1}$.  This coordinate identity does not assert irreducibility or
pairwise coprimality of the Schur factors.

The row span of $H_g^{\rm rem}$ defines a morphism
\begin{equation}\label{eq:grassmann-map}
 \Phi_{n,r}:\A_{\mathbb Z}^r\longrightarrow\Gr(r,n).
\end{equation}
The first $r$ columns form the identity, so $\Phi_{n,r}$ lands in the standard big cell.
It is a closed immersion into that cell \cite[Proposition 2.4]{LiYuanSchur}.
Indeed, $n>r$ and
$Q_r=(-A_0,\ldots,-A_{r-1})^{\mathsf T}$ recover all
coefficient parameters, while the remaining columns are
polynomial in them.
We shall pull back coordinate rank loci of every support size along the same map.

\subsection{Hook--Schur matrices and determinantal resolutions}\label{sec:schur}

\subsubsection{A universal hook remainder formula}

Let $h_m=s_{(m)}$ be the complete homogeneous symmetric function, with $h_0=1$ and
$h_m=0$ for $m<0$.  The coefficient convention \eqref{eq:intro-universal} gives
\begin{equation}\label{eq:h-generating}
 \sum_{m\ge0}h_mT^m=
 \frac{1}{1+A_{r-1}T+A_{r-2}T^2+\cdots+A_0T^r}.
\end{equation}

\begin{lemma}[Hook remainder identity]\label{lem:hook-remainder}
For $m\ge1$ and $0\le b\le r-1$,
\begin{equation}\label{eq:hook-remainder}
 q_{b,r-1+m}=(-1)^{r-1-b}S_{(m,1^{r-1-b})}.
\end{equation}
Equivalently, in $R_r[x]/(G)$,
\begin{equation}\label{eq:hook-polynomial-remainder}
 x^{r-1+m}\equiv
 \sum_{b=0}^{r-1}(-1)^{r-1-b}S_{(m,1^{r-1-b})}x^b.
\end{equation}
The identity is integral and valid after arbitrary base change.
\end{lemma}

\begin{proof}
For algebraically independent roots, bialternant expansion gives
\begin{equation}\label{eq:hook-evaluation-identity}
 X_i^{r-1+m}=
 \sum_{b=0}^{r-1}(-1)^{r-1-b}
 s_{(m,1^{r-1-b})}(X_1,\ldots,X_r)X_i^b
 \qquad(1\le i\le r).
\end{equation}
Equivalently, use Jacobi--Trudi,
$s_{(m,1^j)}=\sum_{v=0}^j(-1)^vh_{m+v}e_{j-v}$, and the root recurrence.
The resulting remainder vanishes at every $X_i$ and hence is zero.
Its coefficients lie in the integral symmetric-polynomial ring, so
Vieta's isomorphism \eqref{eq:vieta} proves the identity over $R_r$
and after arbitrary base change.
\end{proof}

The case $m=1$ reads $x^r\equiv-\sum_{b=0}^{r-1}A_bx^b$, because
$S_{(1^{r-b})}=e_{r-b}=(-1)^{r-b}A_b$.

\subsubsection{Elimination of standard columns}

Fix a nonempty block and use the notation $A,U,B$ of \eqref{eq:intro-AUB}.  Put
$a_0=|A|$ and $u_I=|U|$, so $a_0+u_I=s$ and $|B|=r-a_0$.  Define
\begin{equation}\label{eq:MI-definition}
 M_I=(q_{b,u})_{b\in B,\,u\in U}.
\end{equation}

\begin{proposition}[Hook-Schur reduction]\label{prop:hook-reduction}
After row and column permutations and elementary column operations, $H_I$ has block form
\begin{equation}\label{eq:block-reduction}
 \begin{pmatrix}I_{a_0}&0\\0&M_I\end{pmatrix}.
\end{equation}
Consequently,
\begin{equation}\label{eq:ideal-reduction}
 J_I^{(s)}=I_{u_I}(M_I),
\end{equation}
and its entries are
\[
 (M_I)_{b,u}=(-1)^{r-1-b}S_{(u-r+1,1^{r-1-b})}
 \qquad(b\in B,\ u\in U).
\]
\end{proposition}

\begin{proof}
Move the standard columns $Q_a=e_a$ and their pivot rows first, then
subtract $\sum_{a\in A}q_{a,u}e_a$ from each $Q_u$.  This gives
\eqref{eq:block-reduction}.  Every nonzero $s$-minor contains all pivot
rows and expands to a maximal minor of $M_I$; all such minors occur.
Lemma~\ref{lem:hook-remainder} supplies the entries.
\end{proof}

Minimal graded resolutions use $M_I$; $H_I$ has unit entries when
$A\ne\varnothing$.

\subsubsection{Schur augmentation on the invertible locus}

First suppose that $0\in I$.  For a subset $C\subseteq B$ with $|C|=r-s$, the union
$I\cup C$ has cardinality $r$.  Let $\kappa((I\cup C)\setminus\{0\})$ be the partition associated to this
reduced exponent set by \eqref{eq:kappa-U}.

\begin{theorem}[Schur augmentation]\label{thm:schur-augmentation}
If $0\in I$, then
\begin{equation}\label{eq:schur-augmentation}
 J_I^{(s)}=
 \bigl(S_{\kappa((I\cup C)\setminus\{0\})}:
 C\subseteq B,\ |C|=r-s\bigr).
\end{equation}
For an arbitrary $I$, put $i_0=\min I$ and $I'=I-i_0$.  Then
\begin{equation}\label{eq:schur-localization}
 J_I^{(s)}R_r[A_0^{-1}]
 =J_{I'}^{(s)}R_r[A_0^{-1}],
\end{equation}
so the right-hand side of \eqref{eq:schur-augmentation}, formed with $I'$, gives the defining
ideal on $D(A_0)$.  Equivalently,
\begin{equation}\label{eq:schur-saturation}
 (J_I^{(s)}:A_0^\infty)=(J_{I'}^{(s)}:A_0^\infty),
\end{equation}
and $J_{I'}^{(s)}$ is the displayed Schur ideal because $0\in I'$.
\end{theorem}

\begin{proof}
Expanding the determinant on $I\cup C$ along the standard columns
$A\cup C$ gives the maximal minor of $M_I$ on $B\setminus C$, up to sign.
All maximal minors occur.  Since $0\in I$, \eqref{eq:prior-schur-plucker}
identifies these determinants with the displayed Schur generators;
Proposition~\ref{prop:hook-reduction} proves \eqref{eq:schur-augmentation}.

For general $I$, the Krylov identity gives
$H_I=T_G^{i_0}H_{I'}$.  Since
$\det T_G=(-1)^rA_0$, the matrix $T_G^{i_0}$ and its $s$th compound matrix are invertible after
localizing at $A_0$.  The vectors of $s$-minors therefore differ by an invertible linear
transformation, proving \eqref{eq:schur-localization}.
\end{proof}

Theorem~\ref{thm:schur-augmentation} shows that the maximal Schur--Pl\"ucker system
of \cite{LiYuanSchur} contains every smaller column-rank locus through structured
subfamilies of the same rectangular coordinates.

\subsubsection{Squarefree exterior coordinates}

There is a complementary description over a splitting field.  Suppose
$g=\prod_{\nu=1}^r(x-\lambda_\nu)$ with nonzero pairwise distinct roots.  For an
$s$-element exponent set $I=\{i_0<\cdots<i_{s-1}\}$, put
\begin{equation}\label{eq:kappa-exterior}
 \kappa_s(I)=
 \bigl(i_{s-1}-i_0-(s-1),\ldots,i_1-i_0-1,0\bigr).
\end{equation}

\begin{proposition}[Subset-Schur criterion]\label{prop:subset-schur}
For every $s$-element root subset $J\subseteq\{1,\ldots,r\}$,
\begin{equation}\label{eq:subset-schur-minor}
 \det(\lambda_\nu^{i_j})_{\nu\in J,\,0\le j<s}
 =\left(\prod_{\nu\in J}\lambda_\nu^{i_0}\right)
 \Delta(\Lambda_J)s_{\kappa_s(I)}(\Lambda_J),
\end{equation}
up to the fixed sign determined by the row ordering.  Consequently,
\begin{equation}\label{eq:subset-schur-rank}
 \rank H_I<s
 \quad\Longleftrightarrow\quad
 s_{\kappa_s(I)}(\Lambda_J)=0
 \quad\text{for every }J, |J|=s.
\end{equation}
\end{proposition}

\begin{proof}
Root evaluation preserves rank.  Factoring $\lambda_\nu^{i_0}$ from
each row and applying the bialternant formula proves
\eqref{eq:subset-schur-minor}.  The prefactors are nonzero because the
roots are distinct and nonzero; vanishing of all $s$-minors proves the criterion.
\end{proof}

For $s=r-1$, the coordinates in \eqref{eq:subset-schur-rank} are the deleted-root values
$s_{\kappa} (\Lambda_g\setminus\{\lambda_\nu\})$.  For $s=r$, the exterior space is
one-dimensional and the vector collapses to the single Schur coordinate in
\eqref{eq:prior-schur-plucker}.

\subsubsection{The codimension-two Hilbert--Burch theory}

Take $s=r-1$ and let $u_I=|U|\ge1$.  Since
$a_0+u_I=r-1$, the set $B$ has cardinality $u_I+1$.  For $b\in B$, let $M_I^{\widehat b}$
be obtained by deleting row $b$ and define the signed maximal minor
\begin{equation}\label{eq:Fb}
 F_b=\epsilon_b\det M_I^{\widehat b},
\end{equation}
where the signs are chosen so that $(F_b)_{b\in B}M_I=0$.

Give $R_r$ the positive root-weight grading
\begin{equation}\label{eq:root-weight}
 \rwt(A_{r-j})=j\qquad(1\le j\le r).
\end{equation}
Since Schur functions have root weight equal to their partition size,
\begin{equation}\label{eq:entry-weight}
 \rwt(q_{b,u})=u-b.
\end{equation}
Put $D_I=\sum_{u\in U}u-\sum_{b\in B}b$.

\begin{theorem}[Weighted Schur--Hilbert--Burch resolution]
\label{thm:HB}\label{cor:weighted-HB}
The height-two ideal $J_I^{(r-1)}$ is generated minimally by the $u_I+1$
signed maximal minors $F_b$, of root weights
\begin{equation}\label{eq:F-weight}
 \rwt(F_b)=D_I+b.
\end{equation}
Its integral Hilbert--Burch resolution is
\begin{equation}\label{eq:HB-ungraded}
 0\longrightarrow R_r^{u_I}\xrightarrow{\ M_I\ }R_r^{u_I+1}
 \xrightarrow{\ (F_b)\ }J_I^{(r-1)}\longrightarrow0,
\end{equation}
and its homogeneous form is
\begin{equation}\label{eq:weighted-HB}
 0\longrightarrow\bigoplus_{u\in U}R_r(-(D_I+u))
 \xrightarrow{\ M_I\ }
 \bigoplus_{b\in B}R_r(-(D_I+b))
 \longrightarrow J_I^{(r-1)}\longrightarrow0.
\end{equation}
The field fibers are minimal.  In particular,
\begin{equation}\label{eq:HB-Hilbert-series}
 \Hilb_{R_{r,K}/J_{I,K}^{(r-1)}}(z)=
 \frac{1-\displaystyle\sum_{b\in B}z^{D_I+b}
 +\displaystyle\sum_{u\in U}z^{D_I+u}}
 {\displaystyle\prod_{j=1}^r(1-z^j)}.
\end{equation}
If $0\in I$, then globally
\begin{equation}\label{eq:Fb-schur}
 F_b=\pm S_{\kappa((I\cup\{b\})\setminus\{0\})}.
\end{equation}
\end{theorem}

\begin{proof}
Theorems~\ref{thm:CM-flat} and \ref{thm:schur-augmentation}, together
with Proposition~\ref{prop:hook-reduction}, give height two, the
Hilbert--Burch presentation, and \eqref{eq:Fb-schur}.
Deleting row $b$ gives weight
$\sum_{u\in U}u-\sum_{c\in B\setminus\{b\}}c=D_I+b$;
the shift difference is $u-b>0$.  This proves the homogeneous resolution,
its fiberwise minimality, and its alternating Hilbert series.
\end{proof}

For $r\ge3$, a basic hook block is particularly transparent.  Take
\begin{equation}\label{eq:simple-hook-support}
 I=\{0,1,\ldots,r-3,r\},
\end{equation}
Then $U=\{r\}$ and $B=\{r-2,r-1\}$.  Hence
\begin{equation}\label{eq:simple-hook-ideal}
 J_I^{(r-1)}=(S_{(1,1)},S_{(1)})=(A_{r-2},-A_{r-1}).
\end{equation}
This is a basic linear codimension-two component of the $(r-1)$-sparse threshold locus.

\subsubsection{General Betti numbers and complete intersections}

Return to arbitrary $s$.  Put $h=r-s+1$, $q_I=|B|$, and define anew
\[
 D_I=\sum_{u\in U}u-\sum_{b\in B}b,\qquad
 W_r=\sum_{j=1}^rj=\frac{r(r+1)}2.
\]
For a positively graded Cohen--Macaulay algebra $C$, we use
$a_{\rm rwt}(C)=-\min\{j:(\omega_C)_j\ne0\}$.
We call $C$ \emph{root-weight level} when its canonical module is generated
in one root weight.  Since $a_0+u_I=s$,
\begin{equation}\label{eq:q-ell-h}
 q_I=r-a_0=u_I+h-1.
\end{equation}

\begin{theorem}[Root-weighted Eagon--Northcott resolution]\label{thm:betti}
Assume $N=\max I\ge r$.  Write $u_I=|U|\ge1$, $q_I=|B|=u_I+h-1$,
and put
\[
 E_U=\bigoplus_{u\in U}R_r(-u),\qquad
 F_B=\bigoplus_{b\in B}R_r(-b).
\]
The hook matrix is a homogeneous degree-zero map $E_U\to F_B$.
The maximal-minor quotient has an integral, ground-ring-base-change-compatible
Eagon--Northcott resolution whose term at homological degree $i$ is
\begin{equation}\label{eq:weighted-EN-terms}
 \mathsf{EN}_i(M_I)=
 \bigoplus_{\substack{C\subseteq B,\ |C|=u_I+i-1\\
                    \alpha\in\mathbb N^U,\ |\alpha|=i-1}}
 R_r\bigl(-d_i(C,\alpha)\bigr),\qquad 1\le i\le h,
\end{equation}
where
\begin{equation}\label{eq:weighted-EN-shifts}
 d_i(C,\alpha)=\sum_{u\in U}u+\sum_{u\in U}\alpha_u u-\sum_{b\in C}b.
\end{equation}
Over each field fiber this resolution is minimal in the positive root-weight grading.
Its ungraded Betti numbers are $\beta_0=1$ and
\begin{equation}\label{eq:betti}
 \beta_i=\binom{q_I}{u_I+i-1}\binom{u_I+i-2}{i-1},\qquad1\le i\le h.
\end{equation}
Its full root-weighted Hilbert series is
\begin{equation}\label{eq:weighted-EN-Hilbert}
 \Hilb_{R_{r,K}/J_{I,K}^{(s)}}(t)=
 \frac{1+\displaystyle\sum_{i=1}^h(-1)^i
       \displaystyle\sum_{\substack{|C|=u_I+i-1\\|\alpha|=i-1}}
       t^{d_i(C,\alpha)}}{\displaystyle\prod_{j=1}^r(1-t^j)}.
\end{equation}
The last free term is explicitly
\begin{equation}\label{eq:weighted-EN-top}
 \mathsf{EN}_h(M_I)=
 \bigoplus_{\substack{\alpha\in\mathbb N^U\\|\alpha|=h-1}}
 R_r\!\left(-D_I-\sum_{u\in U}\alpha_u u\right).
\end{equation}
For each field $K$, the canonical module of $C_I=R_{r,K}/J_{I,K}^{(s)}$
has minimal generator root weights
\begin{equation}\label{eq:weighted-canonical-generators}
 W_r-D_I-\sum_{u\in U}\alpha_u u,
 \qquad \alpha\in\mathbb N^U,\quad |\alpha|=h-1.
\end{equation}
Consequently,
\begin{equation}\label{eq:weighted-canonical-invariants}
 \operatorname{type}(C_I)=\binom{u_I+h-2}{h-1},\qquad
 a_{\rm rwt}(C_I)=D_I+(h-1)N-W_r.
\end{equation}
For these field fibers the following conditions are equivalent:
\begin{equation}\label{eq:CI-classification}
 \begin{gathered}
 C_I\text{ is a complete intersection}
 \ \Longleftrightarrow\ C_I\text{ is Gorenstein}\\
 \Longleftrightarrow\ C_I\text{ is root-weight level}
 \ \Longleftrightarrow\ (h=1\text{ or }u_I=1).
 \end{gathered}
\end{equation}
The same numerical condition characterizes complete intersections over
$\mathbb Z$; equivalently, $s=r$ or $I$ has exactly one exponent at least $r$.
\end{theorem}

\begin{proof}
Expected grade follows from Theorem~\ref{thm:CM-flat} and
Proposition~\ref{prop:hook-reduction}.  Eagon--Northcott for
$F_B^*\to E_U^*$ has integral terms
\[
 \bigwedge^{u_I+i-1}F_B^*\otimes\Gamma^{i-1}(E_U)
       \otimes\bigwedge^{u_I}E_U.
\]
The wedge and divided-power bases have degrees
\eqref{eq:weighted-EN-shifts}.  The flat-syzygy argument of
Theorem~\ref{thm:CM-flat} gives ground-ring-base-change exactness;
coefficient specialization still requires expected grade.
Since $u-b>0$, the field differentials have positive weight and are
minimal.  Counting basis indices gives \eqref{eq:betti} and
\eqref{eq:weighted-EN-Hilbert}.

The field-fiber generator number $\binom{u_I+h-1}{u_I}$ equals height
$h$ exactly for $h=1$ or $u_I=1$.  The integral ideal then has one
determinant or $h$ single-column entries, respectively, so expected
grade makes it a complete intersection.  Otherwise it fails this
property at every field fiber's homogeneous vertex.

At $i=h$, necessarily $C=B$, giving \eqref{eq:weighted-EN-top}.
Dualizing against $R_{r,K}(-W_r)$ gives the canonical generators in
\eqref{eq:weighted-canonical-generators}; positive-weight entries
ensure minimality.  Their number gives the type, and their least degree
$W_r-D_I-(h-1)\max U$ gives the $a$-invariant.

The last term has one multiindex exactly when $h=1$ or $u_I=1$.
Otherwise, concentrating all $h-1$ units at the smallest and largest
members of $U$ gives distinct generator degrees and type greater than one.
This proves \eqref{eq:CI-classification}.  Thus root-weight levelness
differs from that of the standard projective compactification.
For $h=2$, the syzygy shifts $D_I+u$ and canonical weights $W_r-D_I-u$
agree with Theorem~\ref{cor:weighted-HB}.
\end{proof}

\section{Relation incidence, contractions, and postulation}\label{v6:app:postulation}

The incidence and contraction interpretations are followed by
standard-graded consequences of Corollary~\ref{cor:projective-resolution}.
These supplements are independent of the reconstruction proof.

\Needspace{9\baselineskip}
\subsection{The Krylov--Toeplitz exact sequence}

Fix $N\ge r$, put $L=N-r+1$, and let
\begin{equation}\label{eq:Toeplitz-matrix}
 \Gamma_{N,r}(G)=\bigl(A_{j-p}\bigr)_
 {0\le p<L,\ 0\le j\le N},
 \qquad A_r=1,\quad A_\nu=0\quad(\nu<0\text{ or }\nu>r).
\end{equation}
Thus the $p$th row is the coefficient vector of $x^pG$.  Write
$H^{\rm rem}_{[0,N]}=(Q_0\ \cdots\ Q_N)$.

\begin{theorem}[Krylov--Toeplitz Fitting duality]\label{thm:Toeplitz-duality}
Over $R_r=\mathbb Z[A_0,\ldots,A_{r-1}]$ there is an exact sequence
\begin{equation}\label{eq:Toeplitz-exact}
 0\longrightarrow R_r^L
 \xrightarrow{\ \Gamma_{N,r}(G)^{\mathsf T}\ }
 R_r^{N+1}
 \xrightarrow{\ H^{\rm rem}_{[0,N]}\ }
 R_r^r\longrightarrow0.
\end{equation}
For $I\subseteq\{0,\ldots,N\}$ of cardinality $1\le s\le r$,
put $J=I^c$. Projection to the $J$-coordinates induces a
canonical isomorphism
\begin{equation}\label{eq:Toeplitz-cokernel}
 \operatorname{coker}H_I\simeq
 \operatorname{coker}\Gamma_{N,r}(G)_J^{\mathsf T}.
\end{equation}
Consequently, scheme-theoretically over $\mathbb Z$,
\begin{equation}\label{eq:Toeplitz-Fitting}
 J_I^{(s)}=I_s(H_I)=I_L\bigl(\Gamma_{N,r}(G)_J\bigr).
\end{equation}
Moreover, after compatible choices of orientations, complementary maximal minors satisfy
\begin{equation}\label{eq:complementary-minors}
 \det\Gamma_{N,r}(G)_D=\pm\det H^{\rm rem}_{D^c},
 \qquad |D|=L.
\end{equation}
In particular,
\begin{equation}\label{eq:global-Plucker-augmentation}
 J_I^{(s)}=
 \bigl(\det H_S^{\rm rem}: I\subseteq S,\ |S|=r\bigr).
\end{equation}
\end{theorem}

\begin{proof}
Monic division identifies the kernel of the remainder map with the
multiples $UG$, $\deg U<L$, proving \eqref{eq:Toeplitz-exact}.
The first $r$ columns of the remainder matrix form the identity,
so this exact sequence is split and remains exact after arbitrary
base change.
Projecting this kernel to the coordinates $J$ has cokernel
$R_r^r/\operatorname{im}H_I$, which proves
\eqref{eq:Toeplitz-cokernel}.  Its $(r-s)$th Fitting ideal is
$I_s(H_I)$ in the remainder presentation and $I_L(\Gamma_J)$ in the
multiplication presentation, since $|J|=L+r-s$.
Taking top exterior powers of the split exact sequence gives
\eqref{eq:complementary-minors}; ranging over the $L$-subsets of $J$
then gives \eqref{eq:global-Plucker-augmentation}.
\end{proof}

The Schur--Pl\"ucker formula of \cite{LiYuanSchur} turns the last presentation into
\begin{equation}\label{eq:global-Schur-augmentation}
 J_I^{(s)}=
 \left(
 A_0^{\min S}S_{\kappa(U_S)}:
 I\subseteq S,\ |S|=r
 \right),
\end{equation}
where $U_S=(S-\min S)\setminus\{0\}$, and harmless signs and powers of $(-1)^r$
have been suppressed.  Thus the localized
augmentation in Theorem~\ref{thm:schur-augmentation} is the minimal form of a global
complementary-minor identity.

\subsection{Monomial compactifications}

The following specializations and numerical consequences use the
uniform family of Theorem~\ref{v6:lin:uniform}.  Only the density of a
chosen affine chart requires an additional monomial-support argument.

Let $1\le s\le r\le N$, let $I=\{i_0<\cdots<i_{s-1}=N\}$, and set
\begin{equation}\label{eq:monomial-linear-series}
 V_{I,\mathbb Z}=\bigoplus_{i\in I}\mathbb Z X^iZ^{N-i}
 \subseteq H^0(\PP^1_{\mathbb Z},\mathcal O(N)).
\end{equation}
We write $V_I$ after the ground ring has been specified.  \begin{theorem}[Integral uniform secant compactification]\label{thm:secant-compactification}
Over $\mathbb Z$, define
\begin{equation}\label{eq:projective-secant-locus}
 \overline X_I^{(s)}=
 D_{s-1}\left(V_{I,\mathbb Z}\otimes\mathcal O_{\PP^r}
 \longrightarrow\cE_{N,r}\right).
\end{equation}
Its monic chart $A_r\ne0$ is scheme-theoretically $X_I^{(s)}$.  If $T_I(G)$
denotes multiplication by the universal binary form followed by the quotient,
\[
 H^0(\mathcal O(N-r))\xrightarrow{\,\cdot G\,}
 H^0(\mathcal O(N))\longrightarrow H^0(\mathcal O(N))/V_{I,\mathbb Z},
\]
then $T_I(G)$ is an $(N-s+1)\times(N-r+1)$ homogeneous linear Toeplitz matrix and
\begin{equation}\label{eq:projective-Toeplitz}
 \overline X_I^{(s)}=\operatorname{Proj}R_{I,\mathbb Z},\qquad
 R_{I,\mathbb Z}=
 \mathbb Z[A_0,\ldots,A_r]/I_{N-r+1}(T_I(G)).
\end{equation}
The ring $R_{I,\mathbb Z}$ is Cohen--Macaulay of dimension $s+1$ and flat over
$\mathbb Z$.  Its integral Eagon--Northcott complex is a free resolution and
remains exact under every ground-ring change $\mathbb Z\to B_0$.
The scheme $\overline X_I^{(s)}$ is flat over $\operatorname{Spec}\mathbb Z$
with pure relative dimension $s-1$.  Each field fiber is arithmetically
Cohen--Macaulay and is the scheme-theoretic closure of its monic chart.
In $A^{r-s+1}(\PP^r_K)$ its fundamental cycle and degree are
\begin{equation}\label{eq:universal-projective-degree}
 [\overline X_{I,K}^{(s)}]
 =\binom{N-s+1}{r-s+1}H^{r-s+1},\qquad
 \deg\overline X_{I,K}^{(s)}=\binom{N-s+1}{r-s+1}.
\end{equation}
\end{theorem}

\begin{proof}
Apply Theorem~\ref{v6:lin:uniform} to the integral subbundle
$V_{I,\mathbb Z}$.  The cokernel presentation gives
\eqref{eq:projective-Toeplitz}; on $A_r\ne0$, restriction to the
universal divisor is monic polynomial division, so its Fitting ideal
is $J_I^{(s)}$ by Theorem~\ref{thm:Toeplitz-duality}.
The general theorem supplies the field height $h$, the
Cohen--Macaulay fibers, integral flatness, and the universally exact
graded resolution.  Flatness and Cohen--Macaulay fibers over the
regular base $\mathbb Z$ make $R_{I,\mathbb Z}$ Cohen--Macaulay
of dimension $s+1$ \cite{EisenbudCA}.  The assertions about Proj follow
by homogeneous localization and passage to degree zero.

It remains to prove monic density, which is specific to $V_I$.
Over an algebraic closure, the incidence
$\{([G],[H]):[GH]\in\PP(V_I)\}$ is finite flat over the integral
space $\PP(V_I)$ by Lemma~\ref{lem:binary-multiplication}.
Every incidence component dominates this base: flatness excludes
minimal primes lying over nonzero base primes.  Since $N\in I$, the
relations with nonzero leading coefficient form a nonempty open set;
neither factor over this set has a root at infinity.  Every image
component therefore meets $A_r\ne0$.  The homogeneous field quotient
is Cohen--Macaulay, so $A_r$ avoids every associated prime and is a
nonzerodivisor.  Its ideal is consequently $A_r$-saturated, proving
the scheme-theoretic closure assertion.  This argument descends to
every field and does not assume affine expected codimension.

Finally, \eqref{eq:secant-bundle-resolution} gives
$c(\cE_{N,r})=(1-H)^{-(N-r+1)}$.
Thom--Porteous in the established expected codimension yields
$[\overline X_{I,K}^{(s)}]=c_h(\cE_{N,r})$
and \eqref{eq:universal-projective-degree}.
\end{proof}

\subsection{Affine fixed-support geometry}\label{sec:fixed}

Retain \eqref{eq:intro-universal}--\eqref{eq:intro-block}; the
superscript in $J_I^{(s)}$ denotes support size.  The relation incidence
gives the affine interpretation and an elementary height check.

\paragraph{Relation incidence.}

Let $K$ be a field and write $R_{r,K}=K[A_0,\ldots,A_{r-1}]$.  Introduce
\begin{equation}\label{eq:incidence}
 Z_{I,K}=
 \left\{(g,[c_i]_{i\in I})\in\A_K^r\times\PP_K^{s-1}:
 g\mid\sum_{i\in I}c_ix^i\right\}.
\end{equation}
Its equations are the $r$ entries of $H_Ic=0$, so $Z_{I,K}$ is closed.

\begin{remark}[Support interpretation]\label{prop:support-interpretation}
For every extension field $L/K$, the following conditions on $g\in\A_K^r(L)$ are equivalent:
\begin{enumerate}[label=\textup{(\roman*)}]
\item $g\in X_{I,K}^{(s)}(L)$;
\item the columns $(Q_i(g))_{i\in I}$ are linearly dependent over $L$;
\item there is a nonzero polynomial $f=\sum_{i\in I}c_ix^i$ divisible by $g$.
\end{enumerate}

The vanishing of all $s\times s$ minors of an $r\times s$ matrix is equivalent to rank less
than $s$, and hence to a nonzero kernel vector $c$.  Since $H_Ic$ is the coefficient vector of
$f\bmod g$, the last two conditions are equivalent.
\end{remark}

The projection $Z_{I,K}\to\A_K^r$ is proper because the second factor is projective.  Its
set-theoretic image is therefore closed and equals $|X_{I,K}^{(s)}|$.  We use the determinantal
scheme structure on this image.

\begin{theorem}[Expected codimension]\label{thm:expected-codimension}
Let $1\le s\le r$ and $I\subseteq\{0,\ldots,n-1\}$ have cardinality $s$.
\begin{enumerate}[label=\textup{(\roman*)}]
\item If $I\subseteq\{0,\ldots,r-1\}$, then $J_I^{(s)}=R_r$ and $X_I^{(s)}=\varnothing$.
\item If $\max I\ge r$, then for every field $K$,
\begin{equation}\label{eq:expected-codimension}
 \operatorname{ht}J_{I,K}^{(s)}=r-s+1,
 \qquad \dim X_{I,K}^{(s)}=s-1.
\end{equation}
\end{enumerate}
\end{theorem}

\begin{proof}
If all indices are below $r$, a maximal minor of the distinct standard
basis columns is $1$.  Otherwise work over an algebraic closure.
Every nonzero $f_c=\sum_{i\in I}c_ix^i$ has finitely many monic
degree-$r$ divisors, so the relation projection has finite fibers and
\begin{equation}\label{eq:incidence-upper}
 \dim X_{I,K}^{(s)}\le\dim Z_{I,K}\le s-1.
\end{equation}
At $g=x^r$ every column of index at least $r$ vanishes, so the ideal
is proper.  The maximal-minor height bound gives
\begin{equation}\label{eq:det-height-bound}
 \operatorname{ht}I_s(H_I)\le r-s+1,
\end{equation}
which proves the opposite dimension inequality
\cite{EagonNorthcott,BrunsVetter}.
\end{proof}

Thus the relation-space dimension $s-1$, not the ambient exponent bound
$n$, controls the dimension of each nonempty block.

\paragraph{Integral determinantal structure.}

Let $\mathsf F=R_r^r$, $\mathsf E=R_r^s$, and regard the transpose of $H_I$ as a map $\mathsf F\to \mathsf E$.  Put
$h=r-s+1$.  The Eagon--Northcott terms resolving the maximal-minor quotient have the form
\begin{equation}\label{eq:EN-terms}
 \mathsf{EN}_i(H_I)=
 \bigwedge^{s+i-1}\mathsf F\otimes\Gamma^{i-1}(\mathsf E^*)\otimes\bigwedge^s(\mathsf E^*),
 \qquad 1\le i\le h.
\end{equation}

Here $\Gamma^j(\mathsf E^*)=(\sym^j \mathsf E)^*$ is the divided-power module.  This dual-symmetric,
rather than ordinary symmetric, convention gives the Eagon--Northcott complex over
$\mathbb Z$ and in arbitrary characteristic.

\begin{theorem}[Universal Cohen--Macaulay and flat structure]\label{thm:CM-flat}
If $\max I\ge r$, the Eagon--Northcott complex
\begin{equation}\label{eq:EN-resolution}
 0\longrightarrow\mathsf{EN}_h(H_I)\longrightarrow\cdots
 \longrightarrow\mathsf{EN}_1(H_I)\longrightarrow R_r
 \longrightarrow R_r/J_I^{(s)}\longrightarrow0
\end{equation}
is exact.  The ring $R_r/J_I^{(s)}$ is Cohen--Macaulay of dimension $s$, is torsion-free and
flat over $\mathbb Z$, and all of its field fibers are Cohen--Macaulay and pure of dimension
$s-1$.
\end{theorem}

\begin{proof}
Theorem~\ref{thm:secant-compactification} identifies this quotient
with the monic chart of the integral compactification, proving
flatness, Cohen--Macaulayness and the stated dimensions; torsion-free
is equivalent to flat over $\mathbb Z$.
The fiber height $h$ in Theorem~\ref{thm:expected-codimension},
including the generic fiber, gives grade $h$ in the regular ring
$R_r$.  Thus Eagon--Northcott gives \eqref{eq:EN-resolution}, and
the same argument over each field gives pure Cohen--Macaulay fibers.
The quotient and free terms are $\mathbb Z$-flat, so the successive
syzygy sequences show that every syzygy is $\mathbb Z$-flat.
The resolution therefore remains exact under every ground-ring
change $\mathbb Z\to B_0$.  Specialization of the coefficient
variables is a different operation and requires the expected grade.
\end{proof}

\subsection{Standard-graded consequences}
\label{r18:sec:standard-postulation}
The standard-graded resolution below starts with equations in degree
$L$ and first syzygies, when present, in degree $L+1$.
These degrees are not the root weights of Appendix~\ref{v6:app:coordinates}.

\begin{corollary}[Uniform pure resolution and level structure]
\label{cor:projective-resolution}
Let $K$ be a field, let $1\le s\le r$, and let $N=\max I\ge r$.
Put $S=K[A_0,\ldots,A_r]$,
$R_I=R_{I,\mathbb Z}\otimes_{\mathbb Z}K=S/I_{\overline X_I}$,
$L=N-r+1$, and $h=r-s+1$.  The ideal $I_{\overline X_I}$ has an
$L$-linear minimal resolution, whereas the quotient $R_I$ has the pure
degree sequence $(0,L,L+1,\ldots,L+h-1)$:
\begin{equation}\label{eq:projective-EN}
 0\longrightarrow S(-(L+h-1))^{\beta_h}\longrightarrow\cdots
 \longrightarrow S(-L)^{\beta_1}\longrightarrow S
 \longrightarrow R_I\longrightarrow0,
\end{equation}
where
\begin{equation}\label{eq:projective-Betti}
 \beta_i=\binom{L+h-1}{L+i-1}\binom{L+i-2}{i-1},
 \qquad 1\le i\le h.
\end{equation}
Its Hilbert series and $h$-vector are
\begin{equation}\label{eq:projective-Hilbert-vector}
 \Hilb_{R_I}(t)=
 \frac{\displaystyle\sum_{j=0}^{L-1}\binom{h+j-1}{h-1}t^j}{(1-t)^s},
 \qquad
 h_j=\begin{cases}
 \binom{h+j-1}{h-1},&0\le j<L,\\
 0,&j\ge L.
 \end{cases}
\end{equation}
The ring is arithmetically level, and
\begin{equation}\label{eq:projective-invariants}
 \operatorname{reg}R_I=L-1=N-r,\qquad
 a(R_I)=L-1-s,\qquad
 \operatorname{type}R_I=\binom{L+h-2}{h-1}=\binom{N-s}{r-s}.
\end{equation}
In particular, it is arithmetically Gorenstein exactly when
$L=1$ or $h=1$, equivalently $N=r$ or $s=r$.
\end{corollary}

\begin{proof}
By Theorem~\ref{v6:lin:uniform}, the linear matrix has size
$(L+h-1)\times L$ and expected grade $h$.
Eagon--Northcott gives \eqref{eq:projective-EN} with ranks
\eqref{eq:projective-Betti}; positive-degree differentials make it minimal.
The resolution numerator satisfies
\begin{equation}\label{eq:projective-EN-numerator}
 1+\sum_{i=1}^{h}(-1)^i\beta_i t^{L+i-1}
 =(1-t)^h\sum_{j=0}^{L-1}\binom{h+j-1}{h-1}t^j.
\end{equation}
Indeed, both sides have constant term $1$ and derivative
$-L\binom{L+h-1}{h-1}t^{L-1}(1-t)^{h-1}$; for the left side use
$(L+i-1)\beta_i=L\binom{L+h-1}{h-1}\binom{h-1}{i-1}$.
This verifies an integer polynomial identity, independently of the
ground field.  Division by $(1-t)^{r+1}=(1-t)^{s+h}$ gives the Hilbert series.

The last free module has one shift.  Dualizing against $S(-r-1)$
gives $\beta_h$ minimal canonical generators in degree $s-L+1$;
minimality is preserved because the dual differential has positive-degree
entries.  This proves levelness, the $a$-invariant and type; the shifts
give regularity $L-1$.  The type is one exactly when $L=1$ or $h=1$.
For $L=1$ the ideal consists of $h$ independent linear forms and the
resolution is Koszul, so this endpoint is included.
\end{proof}

\subsection{Relation-incidence and root-weight geometry}

\begin{proposition}[Gorenstein relation incidence]
\label{prop:relation-incidence-Gorenstein}
Assume $N=\max I\ge r$, and write
\begin{equation*}
 \varpi:Z_{I,K}\longrightarrow\PP_K^{s-1}
\end{equation*}
for the relation projection.  Every geometric fiber of $\varpi$ is finite, so $\varpi$ is
quasi-finite.  Moreover, $Z_{I,K}$ is a pure $(s-1)$-dimensional local complete intersection
of codimension $r$ in $\A_K^r\times\PP_K^{s-1}$ and is therefore Gorenstein.  If
$\mathcal O_Z(1)=\varpi^*\mathcal O_{\PP^{s-1}}(1)$, then
\begin{equation}\label{eq:relation-incidence-canonical}
 \omega_{Z_{I,K}}\simeq\mathcal O_Z(r-s).
\end{equation}
These assertions commute with extension of the ground field.
\end{proposition}

\begin{proof}
Theorem~\ref{thm:expected-codimension} gives quasi-finiteness, and
$(x^r,[x^N])$ gives nonemptiness.  In $Y=\A_K^r\times\PP_K^{s-1}$,
the incidence is cut out by $H_I(g)c$, a section of
$\operatorname{pr}_2^*\mathcal O(1)^{\oplus r}$.
Every irreducible component of the zero scheme has dimension at
least $\dim Y-r=s-1$ by Krull's height theorem, and at most $s-1$
by quasi-finiteness of $\varpi$.  Its local defining ideal therefore
has height $r$ and is generated by $r$ elements in the regular
local rings of $Y$.  These elements form a regular sequence.
Adjunction with $\omega_Y=\operatorname{pr}_2^*\mathcal O(-s)$
gives $\omega_{Z_{I,K}}\simeq\mathcal O_Z(r-s)$.
The argument survives every field extension.
\end{proof}

\begin{remark}[Quasi-finite need not mean finite]
\label{rem:relation-projection-not-finite}
The relation projection $\varpi$ is generally not proper.  For
$r=s=N=2$ and $I=\{0,2\}$, one has $Z_I\simeq\A^1$, while $\varpi$ identifies it with an
affine chart of $\PP^1$.  By contrast, the projection
$p:Z_{I,K}\to X_{I,K}^{(s)}$ is projective, because it is the restriction of the first
projection from $\A^r\times\PP^{s-1}$.
\end{remark}

\begin{proposition}[Root-weight contraction]
\label{prop:root-weight-contraction}
Let $K$ be a field and assume $N=\max I\ge r$.  The action
\begin{equation}\label{eq:root-weight-contraction-action}
 t\cdot A_j=t^{r-j}A_j\qquad(0\le j\le r-1)
\end{equation}
of $\Gm$ on $\A_K^r$ preserves $X_{I,K}^{(s)}$ and extends to an $\A^1$-contraction whose
value at $t=0$ is the vertex $g_0=x^r$.  After base change to an algebraic closure, every
irreducible component of $X_I^{(s)}$ contains $g_0$.  Consequently every nonempty
fixed-support block is geometrically connected, and every geometrically normal block is
geometrically integral.
\end{proposition}

\begin{proof}
The scaling identity $G_{t\cdot A}(x)=t^rG_A(x/t)$ gives $q_{b,i}=0$ for $b>i$ and, for
$b\le i$,
\begin{equation}\label{eq:root-weight-remainder-scaling}
 q_{b,i}(t\cdot A)=t^{i-b}q_{b,i}(A).
\end{equation}
Every minor is root-weight homogeneous, so the ideal is stable.
Positive coordinate weights extend the action to $\A^1$ with limit
$g_0$; the $N$th column vanishes there, placing $g_0$ in the block.
Over an algebraic closure, connectedness of $\Gm$ preserves each
irreducible component $C$.  The closure of every orbit in $C$ contains
$g_0$, so all components meet there.  Finally, the components of a
Noetherian normal scheme are disjoint and open-and-closed; a connected
normal geometric fiber is therefore integral.
\end{proof}

\begin{remark}[Root-weight contraction of all rank strata]
\label{cor:vertex-maximal-defect}
Let $K$ be a field, let $N=\max I\ge r$, and write
\[
 A=I\cap\{0,\ldots,r-1\},\qquad
 U=I\cap\{r,r+1,\ldots,N\},\qquad
 u_I=|U|.
\]
For $1\le\delta\le s$, set $X_I^{[\delta]}=V(I_{s-\delta+1}(H_I))$.
This deep-rank scheme is invariant under the root-weight contraction, and
\[
 X_I^{[\delta]}\ne\varnothing
 \quad\Longleftrightarrow\quad
 \delta\le u_I.
\]
After base change to an algebraic closure, every irreducible component of
$X_{I,\mathrm{red}}^{[\delta]}$ contains $g_0=x^r$.  Consequently
\[
 \dim |X_I^{[\delta]}|=\dim_{g_0}|X_I^{[\delta]}|\quad(1\le\delta\le u_I),
 \qquad
 \max_{g\in X_I^{(s)}}\bigl(s-\rank H_I(g)\bigr)=u_I.
\]
\end{remark}

\begin{proof}
All deep-rank ideals are homogeneous by
\eqref{eq:root-weight-remainder-scaling}.  At $g_0$, the standard
columns indexed by $A$ survive and those indexed by $U$ vanish, giving
rank $s-u_I$.  At every point the standard columns already force rank
at least $s-u_I$.  This proves the nonemptiness and maximal-defect
claims; Proposition~\ref{prop:root-weight-contraction}'s orbit argument
gives the component and dimension assertions.
\end{proof}

\section{Profile degrees and formal local models}\label{sec:formal}

Throughout this appendix, let $K$ be algebraically closed of
characteristic zero and assume
\eqref{eq:general-support-normalization}; in particular,
$2\le s\le r<N$.  Retain the profiles over $U_I^{\rm reg}$ from
Section~\ref{sec:monodromy}, with $E_\pi=(\cF_\pi)_{\rm red}$.
Since $B^{\rm prim}\subseteq\{0,\ldots,m\}$ has $s$ elements,
$m\ge s-1$.
The fixed zero root supplies the full piece's Artin factor;
occupancy controls its reduced image.

\Needspace{10\baselineskip}
\subsection{Individual mixed degrees and the information they retain}
\label{v6:degree:section}

For the characteristic-zero profiles of Section~\ref{sec:monodromy},
we compute mixed degrees in $\eta,\xi$, including contracted profiles.

\subsubsection{A reduced profile and its image degree}
\label{r19:sec:individual-degrees}
For $\pi=(e,\boldsymbol m)$, expand the weights into a list
$w_1,\ldots,w_m$, with $m_\nu$ copies of $w(\nu)$, including zeros.
Abbreviate
\begin{equation}\label{v6:degree:data}
 D_\pi=D(\boldsymbol m),\qquad
 \ell_\pi=\ell(\boldsymbol m)=\#\{i:w_i>0\}.
\end{equation}
Thus $D_\pi$ is the reduced relation-cover degree.

\begin{theorem}[Individual profile mixed degrees]
\label{v6:degree:formula}
For every admissible profile and every $0\le j\le s-1$,
\begin{equation}\label{v6:degree:explicit}
 \boxed{\displaystyle
 \delta_{\pi,j}=\int_{W_\pi}\eta^j\xi^{s-1-j}
  =\frac{D_\pi}{\binom mj d^j}\,
    e_j(w_1,\ldots,w_m).}
\end{equation}
Here $e_j$ denotes the elementary symmetric polynomial and $e_0=1$.
In particular, with $k_\pi=\min\{\ell_\pi,s-1\}$,
\begin{equation}\label{v6:degree:image-degree}
 \deg\overline Y_\pi
  =\frac{D_\pi}{\binom m{k_\pi}d^{k_\pi}}
      e_{k_\pi}(w_1,\ldots,w_m).
\end{equation}
There is no factor $\binom ae$ in these formulas for reduced
components.
\end{theorem}

\begin{proof}
For $j=0$, the assertion is the reduced cover degree $D_\pi$.
Let $1\le j\le s-1$.

\emph{Choosing a regular relation slice.}
Choose nonzero $z_1,\ldots,z_j$ with pairwise distinct $d$th powers
so that
\begin{equation}\label{v6:degree:evaluations}
 F(z_1)=\cdots=F(z_j)=0
\end{equation}
are independent and their common zero space meets the regular relation
open.  Indeed, over the open torus of rank-$j$ evaluation tuples the
relation incidence is an integral projective kernel bundle.  Its
regular-relation open is nonempty by Lemma~\ref{lem:uniform-subset-rank},
using $j$ roots in distinct blocks.  The smooth bundle projection is
open, so such tuples contain a nonempty open and hence a $K$-point,
also when $j=s-1$.

Fix such a tuple and a regular point $[F_0]$ of its linear relation
slice.  Choose $s-1-j$ hyperplanes through $[F_0]$ whose restrictions
cut that slice at $[F_0]$ as a reduced point; none is needed for $j=s-1$.
On $W_\pi$, impose these further hyperplanes together with the $j$
factor hyperplanes
\begin{equation}\label{v6:degree:factor-evaluations}
 G(z_1)=\cdots=G(z_j)=0.
\end{equation}
Since $F(z_i)=G(z_i)H(z_i)$, the section ideal contains all the
chosen relation equations.  Their reduced intersection is $[F_0]$,
so the zero scheme is finite and lies scheme-theoretically in
$q^{-1}([F_0])$, away from the relation boundary.

\emph{Intersection lengths on the reduced and full incidence.}
At a contributing point, $z_i$ is a simple root of $G$ and
$H(z_i)$ is a unit.  In local line-bundle trivializations,
\[
 G(z_i)=F(z_i)/H(z_i)
\]
Thus the section ideal pulls back the maximal ideal of $[F_0]$,
generated by $s-1$ independent parameters.  The reduced profile cover
is \emph{\'etale} here, so each intersection has length one.
The section of
$(p^*\mathcal O_{\PP^r}(1))^{\oplus j}\oplus
(q^*\mathcal O_{\PP(V_I)}(1))^{\oplus(s-1-j)}$
is therefore regular at every point of its zero scheme.  Its length
computes $\eta^j\xi^{s-1-j}\cap[W_\pi]$
\cite[Tag 0FA8, Lemma 42.44.1]{Stacks}; if the zero scheme is empty,
this top Chern class is zero.  No global smoothness or
Cohen--Macaulayness of $W_\pi$ is used.

In the full incidence, \eqref{eq:factorization-e-splitting} identifies
the completed local ring at the same point with
$\mathcal A_{a,e}[[v_1,\ldots,v_{s-1}]]$, where the $v_i$ are
relation parameters.  The same $H(z_i)$ remain units.  Freeness over
$K[[v_1,\ldots,v_{s-1}]]$ makes these parameters a regular sequence,
with quotient $\mathcal A_{a,e}$ of length $\binom ae$.

\emph{Counting the reduced points.}
Label the $m$ nonzero blocks of $F_0$ and their $d$ cyclic root
positions.  The wreath-product action on these labels preserves the profile
subset orbit and acts transitively on the $(m)_jd^j$ ordered
$j$-tuples of roots in distinct blocks.  Each root subset of profile $\pi$ contains precisely
$j!e_j(w_1,\ldots,w_m)$ such tuples.  Double-counting pairs of a
profile subset and a contained tuple, and using transitivity,
shows that the number of profile subsets containing the fixed
tuple is
\[
 \frac{D_\pi j!e_j(w_1,\ldots,w_m)}{(m)_j d^j}
 =\frac{D_\pi e_j(w_1,\ldots,w_m)}{\binom mj d^j}.
\]
This is the length of the intersection just constructed.  If
$j>\ell_\pi$, both the intersection and $e_j(w)$ are zero, so
the argument also covers contracted profiles.

\emph{Passing to the image degree.}
Lemma~\ref{lem:profile-rank-image} and
Theorem~\ref{thm:dominant-profile-components} give
$\dim\overline Y_\pi=k_\pi$ and generic fiber
$\PP^{s-1-k_\pi}_{K(\overline Y_\pi)}$ for $p|_{W_\pi}$.
On this fiber the relation map is a linear inclusion, so $\xi$
restricts to $\mathcal O(1)$ of top degree one, also for $k_\pi=0$
\cite[Tags 02TW and 02R4]{Stacks}.  Thus the coefficient in the
proper pushforward is one; cycles over proper closed subsets of the
image cannot contribute in dimension $k_\pi$.  Consequently
\[
 p_*\bigl(\xi^{s-1-k_\pi}\cap[W_\pi]\bigr)
 =[\overline Y_\pi].
\]
The projection formula gives
$\delta_{\pi,k_\pi}=\deg\overline Y_\pi$, proving
\eqref{v6:degree:image-degree}.
\end{proof}

\begin{corollary}[Weight recovery and variance]
\label{v6:degree:recovery}
If $s\ge3$, then
\begin{equation}\label{v6:degree:variance}
 \delta_{\pi,1}^2-\delta_{\pi,0}\delta_{\pi,2}
 =\frac{D_\pi^2}{m^2(m-1)d^2}
   \left(m\sum_{i=1}^m w_i^2-
                   \left(\sum_{i=1}^m w_i\right)^2\right).
\end{equation}
In particular, this defect is zero exactly when all $m$ weights,
including zero weights, are equal.

For a full progression $B^{\rm prim}=\{0,1,\ldots,m\}$, the entire
mixed-degree vector recovers the expanded weight multiset by
\begin{equation}\label{v6:degree:generating-polynomial}
 \boxed{\displaystyle
 \prod_{i=1}^m(1+w_i t)
 =\sum_{j=0}^m\binom mj d^j
      \frac{\delta_{\pi,j}}{D_\pi}t^j.}
\end{equation}
For arbitrary primitive support, the vector determines the
elementary symmetric functions appearing on the right through
degree $s-1$; no recovery of the remaining coefficients is asserted.
\end{corollary}

\begin{proof}
The identities
$e_1(w)=\sum_iw_i$ and
$2e_2(w)=(\sum_iw_i)^2-\sum_iw_i^2$, substituted in
\eqref{v6:degree:explicit}, give \eqref{v6:degree:variance}.
The numerator there also equals
$\sum_{i<j}(w_i-w_j)^2$, giving the equality criterion.
For a full progression, $s=m+1$, so
\eqref{v6:degree:explicit} supplies all coefficients of
$\prod_i(1+w_it)$.  Factoring that polynomial recovers every
positive weight with multiplicity, and the known number $m$
recovers the number of zero weights.  The same formula identifies
the available initial coefficients for an arbitrary support.
\end{proof}

\begin{remark}\label{v6:degree:scope}
These are numerical invariants of the reduced incidence with its specified
classes $\eta,\xi$.  Recovering all weights in
\eqref{v6:degree:generating-polynomial} need not recover necklace shape;
Section~\ref{v6:norm:section} examines that remaining information.
The factor $\binom ae$ enters only when these reduced cycles are summed
inside the full incidence, as follows.
\end{remark}

\subsubsection{Summing reduced profiles and restoring scheme multiplicities}
\label{r19:sec:cycle-conservation}
Summing first at fixed $e$, then weighting by $\binom ae$, recovers
the full incidence cycle, including its boundary.

\begin{corollary}[Two-stage profile mixed-degree conservation]
\label{thm:profile-multidegrees}
Let $K$ be algebraically closed of characteristic zero and assume
\eqref{eq:general-support-normalization}.  Retain $W_\pi$, $\overline Y_\pi$, $k_\pi$, and $\delta_{\pi,j}$ from
Theorems~\ref{thm:dominant-profile-components} and
\ref{v6:degree:formula}.  Put $b_\pi=\binom ae$.  Then
\begin{equation}\label{eq:profile-cycle-decomposition}
 [\overline{\mathfrak Z}_I]=\sum_\pi b_\pi[W_\pi],
 \qquad \deg(q|_{W_\pi})=D(\boldsymbol m).
\end{equation}
The first equality is in $A_{s-1}(\overline{\mathfrak Z}_I)$.
The reduced mixed degrees satisfy
\[
 \delta_{\pi,0}=D(\boldsymbol m),\qquad
 \delta_{\pi,j}=0\ (j>k_\pi),\qquad
 \delta_{\pi,k_\pi}=\deg\overline Y_\pi,
\]
and in $A_{k_\pi}(\PP^r)$ one has
\begin{equation}\label{eq:profile-image-pushforward}
 p_*\bigl(\xi^{s-1-k_\pi}\cap[W_\pi]\bigr)
 =[\overline Y_\pi].
\end{equation}

For each $e$ in \eqref{eq:e-range-all}, with $q_e=r-e$, let
\begin{equation}\label{eq:fixed-e-incidence}
 \begin{aligned}
 Z'_e&=\{([G'],[H'])\in\PP^{q_e}\times\PP^{M-q_e}:
                [G'H']\in\PP(V_{dB^{\rm prim}})\},\\
 \iota_e([G'],[H'])&=([X^eG'],[X^{a-e}H']).
 \end{aligned}
\end{equation}
Here $V_{dB^{\rm prim}}$ consists of binary forms of degree $M$, and
$Z'_e$ has its scheme-theoretic inverse-image structure under binary
multiplication.  It is a reduced complete intersection of dimension
$s-1$, and its irreducible components map isomorphically to the
$W_\pi$ with $e(\pi)=e$.  In
$A_{s-1}(\PP^r\times\PP^{N-r})$, the two-stage cycle formulas are
\begin{equation}\label{eq:fixed-e-cycle-decomposition}
 \begin{aligned}
 [\overline{\mathfrak Z}_I]
   &=\sum_e\binom ae(\iota_e)_*[Z'_e],\\
 (\iota_e)_*[Z'_e]
   &=H_G^eH_H^{a-e}(H_G+H_H)^{M-s+1}.
 \end{aligned}
\end{equation}
On the second line a Chow polynomial denotes its cap product with the
fundamental class of the ambient product.  Consequently
\begin{equation}\label{eq:fixed-e-multidegree-conservation}
 \boxed{\displaystyle
 \sum_{\pi:e(\pi)=e}\delta_{\pi,j}
     =\binom{M-j}{q_e-j}\quad(0\le j\le s-1).}
\end{equation}
Its endpoints are
\begin{equation}\label{eq:fixed-e-endpoint-conservation}
 \begin{aligned}
 \sum_{\pi:e(\pi)=e}D(\boldsymbol m)&=\binom M{q_e},\\
 \sum_{\substack{\pi:e(\pi)=e\\\ell(\boldsymbol m)\ge s-1}}
                  \deg\overline Y_\pi
     &=\binom{M-s+1}{q_e-s+1}.
 \end{aligned}
\end{equation}
Weighting the fixed-$e$ vectors gives
\begin{equation}\label{eq:profile-multidegree-conservation}
 \boxed{\displaystyle\sum_\pi\binom ae\delta_{\pi,j}
 =\sum_e\binom ae\binom{M-j}{r-e-j}
 =\binom{N-j}{r-j}\quad(0\le j\le s-1).}
\end{equation}
In particular,
\begin{equation}\label{eq:profile-endpoint-conservation}
 \sum_\pi\binom ae D(\boldsymbol m)=\binom Nr,
 \qquad
 \sum_{\pi\ \mathrm{dominant}}\binom ae\deg\overline Y_\pi
 =\binom{N-s+1}{r-s+1}.
\end{equation}
Moreover,
$p_*[\overline{\mathfrak Z}_I]=[\overline X_I^{(s)}]$
in $A_{s-1}(\PP^r)$.  For a dominant profile, the recovered-relation
rational map $\rho_\pi:\overline Y_\pi\dashrightarrow\PP(V_I)$ has
degree $D(\boldsymbol m)$.
\end{corollary}

\begin{proof}
\emph{The reduced pieces with fixed $e$.}
Theorem~\ref{thm:dominant-profile-components} gives the component degrees
and generic lengths, hence \eqref{eq:profile-cycle-decomposition}.
The reduced mixed-degree identities and
\eqref{eq:profile-image-pushforward} are those of
Theorem~\ref{v6:degree:formula}.

By Lemma~\ref{lem:binary-multiplication}, $Z'_e$ is finite flat over
$\PP(V_{dB^{\rm prim}})$.  Flat pullback of its $M-s+1$ linear equations
gives a regular sequence of bidegrees $(1,1)$, so $Z'_e$ is
Cohen--Macaulay of dimension $s-1$.  Every component dominates relation
space.  The squarefree, nonzero-constant-term open contains $x^M-1$;
over it multiplication is \emph{\'etale}.  Thus $Z'_e$ is generically
reduced, and $R_0$ together with $S_1$ makes it reduced.

The map $\iota_e$ is a product of linear closed immersions preserving
both hyperplane bundles.  Over the regular relation open its components
are exactly the profiles with split $e$.  Taking reduced closures
identifies them scheme-theoretically with the corresponding $W_\pi$;
no boundary component is omitted since every component dominates the
base.  This includes $q_e=0,M$, with $\PP^0$ and zero hyperplane class
on the degree-zero factor.

\emph{Cycle classes and the two sums.}
The complete-intersection class of $Z'_e$ pushes forward to
$H_G^eH_H^{a-e}(H_G+H_H)^{M-s+1}$.  Together with the generic lengths,
this proves \eqref{eq:fixed-e-cycle-decomposition}.  Since $\iota_e$
preserves $\eta,\xi$,
\[
 \sum_{\pi:e(\pi)=e}\delta_{\pi,j}
 =\int_{\PP^{q_e}\times\PP^{M-q_e}}
        H_G^j(H_G+H_H)^{M-j}
 =\binom{M-j}{q_e-j},
\]
including zero when $j>q_e$.  The endpoints follow from the reduced mixed-degree identities.  The binomial convention makes all terms
outside \eqref{eq:e-range-all} zero.  Weighting by $\binom ae$ and
using Vandermonde's identity gives \eqref{eq:profile-multidegree-conservation}
and \eqref{eq:profile-endpoint-conservation}; equivalently,
\begin{equation}\label{eq:profile-conservation-intersection}
 \sum_\pi b_\pi\delta_{\pi,j}
 =\int_{\PP^r\times\PP^{N-r}}H_G^j(H_G+H_H)^{N-j}
 =\binom{N-j}{r-j}.
\end{equation}
For the final pushforward, Theorem~\ref{thm:dominant-profile-components}
identifies incidence and determinantal generic multiplicities.
Dominant reduced components push forward birationally and contracted
ones to zero.  For a dominant profile, the same birationality identifies
the degree of the recovered-relation map with
$[K(W_\pi):K(\PP(V_I))]=D(\boldsymbol m)$.
\end{proof}

For deeper rank components and contraction, see
Theorem~\ref{thm:deep-rank-profile-components} and
Example~\ref{ex:occupancy-filter}.

\begin{remark}[The incidence is not generally a resolution]\label{rem:incidence-not-resolution}
A contracted profile is still a top-dimensional component of the
finite flat incidence.  Thus the full incidence is not generally a
resolution of $X_I^{(s)}$.  Such a statement would require removing
contracted components and separately proving smoothness of the source.
The fixed-$e$ cycle decomposition is not a global direct-sum
decomposition of schemes: profile closures can meet over the boundary
$t_a\Disc(\overline F_t)=0$.
\end{remark}

\Needspace{10\baselineskip}
\subsection{Fixed zero-root factors and formal neighborhoods}
\label{r26:formal:local}

\begin{proposition}[Formal profile model]\label{prop:formal-profile-model}
\label{cor:generic-profile-invariants}
Assume that $K$ is algebraically closed of characteristic zero.  For every profile $\pi$,
$E_\pi\to U_I^{\rm reg}$ is a connected finite \emph{\'etale} cover of degree
$D(\boldsymbol m)$, and
\begin{equation}\label{eq:profile-product}
 \cF_\pi\simeq\Fact_e(x^a/K)\times_KE_\pi.
\end{equation}
Put
\begin{equation}\label{eq:zero-factor-algebra}
 A_{a,e}=\mathcal O(\Fact_e(x^a/K))
 \simeq H^*(\Gr(e,a),K).
\end{equation}
Here $H^*(\Gr(e,a),K)$ denotes
$H^*(\Gr(e,\mathbb C^a),\mathbb Z)\otimes_{\mathbb Z}K$
with its usual Borel presentation; no embedding
$K\hookrightarrow\mathbb C$ is required.
At every closed geometric point $\widetilde g\in E_\pi$,
\begin{equation}\label{eq:profile-formal-completion}
 \widehat{\mathcal O}_{\cF_\pi,(x^e,\widetilde g)}
 \simeq A_{a,e}[[z_1,\ldots,z_{s-1}]].
\end{equation}
If $\ell(\boldsymbol m)\ge s-1$, there are an open subscheme $V$ of
$X_I^{(s)}$, whose underlying support is a dense open of the reduced image
$Y_\pi=\overline Y_\pi\cap\A^r$, and an open subscheme of $\cF_\pi$
mapped isomorphically onto $V$ by the forgetful incidence map.
Thus, at every closed geometric point $g$ in a suitable dense open of
$Y_\pi$, one has
\begin{equation}\label{eq:profile-image-completion}
 \widehat{\mathcal O}_{X_I^{(s)},g}
 \simeq H^*(\Gr(e,a),K)[[z_1,\ldots,z_{s-1}]].
\end{equation}
In particular, put $p_0=\min(e,a-e)$.  At such a point the local ring is
a complete intersection of embedding codimension $p_0$, and
\begin{equation}\label{eq:profile-local-invariants}
 \dim T_gX_I^{(s)}=(s-1)+p_0,\qquad
 e_g(X_I^{(s)})=\binom ae.
\end{equation}
It is Gorenstein, and it is reduced, equivalently smooth over $K$, exactly
when $e=0$ or $e=a$.  Thus fixed-$e$ incidence branches have the same
general formal model, irrespective of their global degrees and contractions;
for dominant profiles it describes the ambient determinantal local rings
on the stated dense open.
\end{proposition}

\begin{proof}
\emph{Full piece and dominant chart.}
Coprime decomposition \eqref{eq:factorization-e-splitting} and
Theorem~\ref{thm:dominant-profile-components} give $E_\pi$ and the product.
The integral zero-fiber calculation \eqref{eq:cohomology-fiber}
gives \eqref{eq:zero-factor-algebra} \cite{LaksovThorup}:
for $e=1$ it is $K[u]/(u^a)$, and for $e=0,a$ it is $K$,
including $a=0$.  Since $E_\pi$ is \'etale over a smooth base of
dimension $s-1$, completion gives \eqref{eq:profile-formal-completion}.
For dominant profiles, rank-$(s-1)$ minor elimination identifies the
full incidence with the determinantal scheme.  Remove other image
components and restrict the recovered relation to $U_I^{\rm reg}$;
this gives $V$ and \eqref{eq:profile-image-completion}, retaining nilpotents.

\emph{Local invariants.}
Put $q_0=\max(e,a-e)$.  Exchanging the two factors when necessary,
the same zero-fiber presentation gives, for $p_0>0$,
\[
 A_{a,e}\simeq K[c_1,\ldots,c_{p_0}]/(h_{q_0+1},\ldots,h_{q_0+p_0}).
\]
This Artin complete intersection has length $\binom ae$.
Variable weights are at most $p_0\le q_0$, whereas relation weights
exceed $q_0$; thus there are no linear terms and the embedding dimension
is $p_0$.  Adjoining the formal parameters preserves its length as
Hilbert--Samuel multiplicity.  It is $K$ exactly for $e=0,a$;
otherwise it is nonreduced.  Completion detects regularity and the
complete-intersection property, which implies Gorensteinness.
\end{proof}

Equation~\eqref{eq:profile-image-completion} describes the full
$X_I^{(s)}$.  The finer problem of recovering the reduced image algebra across
the boundary is beyond the scope of the present paper.

\subsection{Occupancy counts and bound-attaining rank loci}
\label{r26:formal:deep}

Rotation cycle lengths give the binary-necklace count
\begin{equation}\label{eq:necklace-number}
 a_{d,j}=\frac1d\sum_{d_0\mid\gcd(d,j)}\varphi(d_0)
          \binom{d/d_0}{j/d_0}
\end{equation}
for length $d$ and weight $j$.  Taking multisets of these necklace types gives
\begin{equation}\label{eq:profile-generating-function}
 P_d(u,z,v)=\frac1{1-u}\prod_{j=1}^d(1-uvz^j)^{-a_{d,j}}.
\end{equation}
Here $u$ counts blocks, $z$ selected nonzero roots, and $v$ occupied
blocks; the empty necklace contributes $(1-u)^{-1}$.
By Theorem~\ref{thm:dominant-profile-components}, the number of reduced
image components with fixed split $e$ is
\begin{equation}\label{eq:component-count}
 \sum_{\ell\ge s-1}[u^mz^{r-e}v^\ell]P_d(u,z,v),
\end{equation}
and summing over admissible $e$ counts all components, without Artin
length weights.  For $1\le\delta\le s$, define the deeper determinantal
locus and its reduction
\begin{equation}\label{eq:deep-rank-locus}
 X_I^{[\delta]}=V(I_{s-\delta+1}(H_I)),\qquad
 X_{I,\mathrm{red}}^{[\delta]}=(X_I^{[\delta]})_{\mathrm{red}}
 =\{g:\rank H_I(g)\le s-\delta\}_{\mathrm{red}}.
\end{equation}
Thus $X_{I,\mathrm{red}}^{[1]}=(X_I^{(s)})_{\mathrm{red}}$.

\begin{theorem}[Bound-attaining deep-rank components]
\label{thm:deep-rank-profile-components}
Retain the characteristic-zero hypotheses and notation of
\eqref{eq:general-support-normalization}--\eqref{eq:profile-occupancy}.
For $1\le\delta\le s$,
\begin{equation}\label{eq:deep-rank-dimension-bound}
 \dim X_{I,\mathrm{red}}^{[\delta]}\le s-\delta.
\end{equation}
For $2\le\delta\le s$, let $\Pi_\delta$ be the set of admissible profiles
$\pi=(e,\boldsymbol m)$ with
\begin{equation}\label{eq:deep-profile-occupancy}
 \ell(\boldsymbol m)=s-\delta.
\end{equation}
Taking the closure of the reduced profile image gives a bijection
\begin{equation}\label{eq:deep-profile-bijection}
 \Pi_\delta\ \longleftrightarrow
 \left\{Y\in\Irr(X_{I,\mathrm{red}}^{[\delta]}):
                 \dim Y=s-\delta\right\}.
\end{equation}
Every such component has generic rank exactly $s-\delta$, and the projective relation fiber
over its generic point is $\PP^{\delta-1}$.  In particular, the number of components attaining
the universal bound is
\begin{equation}\label{eq:deep-component-count}
 \sum_{e=\max\{0,r-M\}}^{\min\{a,r\}}
 [u^mz^{r-e}v^{s-\delta}]P_d(u,z,v).
\end{equation}
\end{theorem}

\begin{proof}
\emph{The dimension bound.}
For an irreducible closed $Y\subseteq X_{I,\mathrm{red}}^{[\delta]}$,
choose a dense constant-rank open $U$ of its reduced projective closure,
with rank $k\le s-\delta$.  Lemma~\ref{lem:profile-rank-image} makes
$p^{-1}(U)$ an integral projective bundle of relative dimension $s-k-1$.
Incidence purity (Theorem~\ref{thm:projective-factorization-CI}) gives
\begin{equation*}
 \dim Y+s-k-1\le s-1,\qquad \dim Y\le k\le s-\delta.
\end{equation*}
This proves \eqref{eq:deep-rank-dimension-bound}.

\emph{Attaining the bound.}
If $2\le\delta\le s$ and $\dim Y=s-\delta$, then $k=s-\delta$.
The reduced closure $W$ of $p^{-1}(U)$ is irreducible of dimension $s-1$,
hence is a component of $(\overline{\mathfrak Z}_I)_{\rm red}$.
Properness gives $p(W)=\overline Y$.
Theorem~\ref{thm:dominant-profile-components} now identifies $W=W_\pi$
for a unique admissible profile, with
\begin{equation*}
 \min\{\ell(\boldsymbol m),s-1\}=s-\delta\le s-2,
\end{equation*}
so $\ell(\boldsymbol m)=s-\delta$.

Conversely, occupancy $s-\delta$ gives image dimension $s-\delta$.
Its closed, nonempty monic chart lies in the deep-rank locus, hence is
a component by the bound.  Distinct profiles give distinct projective
images and monic charts, proving \eqref{eq:deep-profile-bijection}.
The generic relation fiber is $\PP^{\delta-1}$; coefficient extraction
in \eqref{eq:profile-generating-function} gives the count.
\end{proof}

\begin{remark}[Scope of the deep-rank classification]
\label{rem:bound-attaining-essential}
Only components of dimension $s-\delta$ are classified.
Even when \eqref{eq:deep-component-count} vanishes, secondary degeneracy
inside profile branches may give nonempty residual components.
Their union has dimension at most $s-\delta-1$; no classification
of them is asserted.
\end{remark}

\Needspace{10\baselineskip}
\subsection{Primitive specializations and a contracted example}
\label{r26:formal:specializations}

\begin{corollary}[Primitive supports and their mixed multidegrees]
\label{cor:primitive-components}
Assume that $K$ is algebraically closed of characteristic zero and that
\eqref{eq:general-support-normalization} has $d=1$.  Put $M=N-a$ and
$q_e=r-e$.  For every integer $e$ in the admissible range
$\max\{0,r-M\}\le e\le\min\{a,r\}$, let $W_e$ be the reduced
profile incidence component and $\overline Y_e$ its coefficient-space
image.  Multiplication by the fixed powers of $X$ identifies $W_e$ with
the integral incidence
\begin{equation}\label{eq:primitive-incidence-model}
 \begin{split}
 Z_e'&=\{([G'],[H'])\in\PP^{q_e}\times\PP^{M-q_e}:
              G'H'\in\PP(V_{B^{\rm prim}})\},\\
 ([G'],[H'])&\longmapsto([X^eG'],[X^{a-e}H']).
 \end{split}
\end{equation}
Here $V_{B^{\rm prim}}$ consists of binary forms of degree $M$.
Writing $H_G,H_H$ for the two hyperplane classes on the smaller
product, one has $[Z_e']=(H_G+H_H)^{M-s+1}$ and
\begin{equation}\label{eq:primitive-mixed-degrees}
 \delta_{e,j}=\binom{M-j}{q_e-j}
 \quad(0\le j\le\min\{q_e,s-1\}),\qquad
 \delta_{e,j}=0\quad(q_e<j\le s-1).
\end{equation}
The dominant images, which are exactly the irreducible components of
$(\overline X_I^{(s)})_{\rm red}$, are indexed by
\begin{equation}\label{eq:primitive-e-range}
 \max\{0,r-M\}\le e\le\min\{a,r-s+1\}.
\end{equation}
For these indices the recovered-relation rational map $\rho_e$ satisfies
\begin{equation}\label{eq:primitive-projective-degrees}
 \deg\overline Y_e=\binom{M-s+1}{q_e-s+1},\qquad
 \deg\rho_e=\binom M{q_e},\qquad
 \operatorname{mult}_{\zeta_{\overline Y_e}}\overline X_I^{(s)}
       =\binom ae.
\end{equation}
If $q_e\le s-1$, then
$\overline Y_e=\PP(X^e\operatorname{Sym}^{q_e}K^2)$ is a linear
space of dimension $q_e$ and degree one.  In particular, if $a=0$,
then $X_I^{(s)}$ and $\overline X_I^{(s)}$ are integral and
$\deg\rho_0=\binom Nr$.
\end{corollary}

\begin{proof}
For $d=1$ each admissible $e$ has one profile of occupancy $q_e$.
Corollary~\ref{thm:profile-multidegrees} identifies the integral $Z'_e$
with $W_e$, including $q_e=0,M$.  The fixed-$e$ conservation sum
\eqref{eq:fixed-e-multidegree-conservation} is therefore its individual
vector.  The occupancy criterion gives \eqref{eq:primitive-e-range}
and the separate image, relation, and multiplicity formulas.

For $q_e<s$, the multiple space
$G'\operatorname{Sym}^{M-q_e}K^2$ has codimension $q_e$, so it
intersects the $s$-dimensional $V_{B^{\rm prim}}$ nontrivially.
Every $[G']$ occurs, giving the full linear image.
For $a=0$, use Corollary~\ref{cor:geometric-reducedness} and $e=0$.
\end{proof}

\begin{example}[Orbit degree, contraction, and image degree]
\label{ex:occupancy-filter}
Take $I=\{0,3,6\}$ and $r=s=3$, so $a=0$, $d=3$, and $m=2$.
The two degree-three selection profiles are $\{000,111\}$ and
$\{100,110\}$.  The first has relation degree $D=2$, occupancy $k=1$,
and reduced projective incidence
\[
 \PP^1\times\PP^1\longrightarrow\PP^3\times\PP^3,
 \qquad
 ([u_0:u_1],[v_0:v_1])\longmapsto
 ([u_0X^3+u_1Y^3],[v_0X^3+v_1Y^3]).
\]
If $U,V$ denote its hyperplane classes, then $\eta=U$,
$\xi=U+V$, and $U^2=V^2=0$, $\int UV=1$.  Hence its mixed vector is
\[
 (\delta_0,\delta_1,\delta_2)
 =\left(\int(U+V)^2,\int U(U+V),\int U^2\right)=(2,1,0).
\]
Its image is the line $A_1=A_2=0$ of cubic forms spanned by $X^3,Y^3$;
the generic relation fiber is $\PP^1$.  This is the unique component of
$X_{I,\mathrm{red}}^{[2]}$ attaining the bound $s-2=1$:
it is the only profile of occupancy one, and
Theorem~\ref{thm:deep-rank-profile-components} applies.  On the monic chart, fixing $G=x^3-z$
and varying $H=x^3-w$ gives
$GH=x^6-(z+w)x^3+zw$ on the same support.  This also exhibits a positive-dimensional affine fiber.

The other profile, $\{100,110\}$, has $D=2\cdot3\cdot3=18$ and
occupancy $k=2$, so its projection is birational onto the unique surface
component.  Since $a=0$, both incidence cycle multiplicities are one.
Corollary~\ref{thm:profile-multidegrees} gives the total vector
\[
 \left(\binom63,\binom52,\binom41\right)=(20,10,4).
\]
Subtracting the directly computed contracted vector yields
\begin{equation}\label{eq:occupancy-example-multidegrees}
 (2,1,0)+(18,9,4)=(20,10,4).
\end{equation}
In particular the dominant image is a quartic surface, whereas its
recovered-relation rational map has degree eighteen.

These degrees agree with the explicit equation.  On the monic chart
$g=x^3+A_2x^2+A_1x+A_0$, reducing $1,x^3,x^6$ modulo $g$ gives
\begin{equation}\label{eq:occupancy-example-equation}
 A_1^2(A_2^2-A_1)-A_2^3A_0=0.
\end{equation}
In homogeneous coefficient coordinates
$G=A_3X^3+A_2X^2Y+A_1XY^2+A_0Y^3$, its closure is
\[
 A_1^2A_2^2-A_1^3A_3-A_2^3A_0=0.
\]
This polynomial is primitive and linear in $A_0$ over
$K[A_1,A_2,A_3]$, since its leading and constant coefficients are
coprime; hence it is irreducible.  The contracted line lies on this
quartic and is not an additional component.  Thus the two relation
branches have degrees $2,18$, while their images are a line and the
single quartic component.
\end{example}

\section{Auxiliary factorization and zero-fiber inputs}\label{app:auxiliary-inputs}
For $2\le s\le r<N$, put
\begin{equation}\label{eq:canonical-I}
 I_{s,N}=\{0,1,\ldots,s-2,N\},\qquad X_{s,N}=X^{(s)}_{I_{s,N}}.
\end{equation}
\subsection{Canonical hook ideal}

Put $m=N-r+1$ and $h=r-s+1$.  For the support \eqref{eq:canonical-I}, the reduced hook matrix
of Proposition~\ref{prop:hook-reduction} is a single column: its rows are indexed by
$b=s-1,\ldots,r-1$.

\begin{proposition}[Canonical hook ideal]\label{prop:canonical-hook-ideal}
There are integral scheme-theoretic equalities
\begin{equation}\label{eq:canonical-hook-h}
 \begin{split}
 J_{I_{s,N}}^{(s)}
 &=\bigl(S_{(m)},S_{(m,1)},\ldots,S_{(m,1^{h-1})}\bigr)\\
 &=(h_m,h_{m+1},\ldots,h_{m+h-1})
  =(h_{N-r+1},\ldots,h_{N-s+1}).
 \end{split}
\end{equation}
In particular, $X_{s,N}$ is a complete intersection over $\mathbb Z$ of codimension $h$ in
the coefficient affine space.
\end{proposition}

\begin{proof}
Hook reduction gives the first equality.  Jacobi--Trudi gives
\begin{equation}\label{eq:hook-triangular}
 S_{(m,1^j)}=\sum_{i=0}^j(-1)^i e_{j-i}h_{m+i}\qquad(0\le j<h).
\end{equation}
The signed change is unitriangular over $\mathbb Z$.
Theorems~\ref{thm:expected-codimension} and \ref{thm:CM-flat} give
height $h$ in regular $R_r$, so these $h$ generators form a regular sequence.
\end{proof}

Section~\ref{sec:monodromy} separately addresses geometric irreducibility.

\subsection{Sparse factorization cover and the integral Schur basis}
\label{r25:factorization:monic-basis}

Lemma~\ref{lem:binary-multiplication} gives finite flat Gorenstein
factorization charts after arbitrary coefficient base change.
The additional integral basis theorem gives
$\{S_\lambda:\lambda\subseteq(N-r)^r\}$ for every monic degree-$N$
polynomial over any commutative ring, $1\le r<N$
\cite[Theorem~0.6 and Corollary~0.7]{LaksovThorup}.
For the canonical slice put
\begin{equation}\label{eq:sparse-base-polynomial}
 B_{s,N}=\mathbb Z[t_0,\ldots,t_{s-2}],\qquad
 F_{s,N}(x)=x^N-\sum_{j=0}^{s-2}t_jx^j.
\end{equation}
The factorization algebra $\Fact_r(F_{s,N}/B_{s,N})$ represents factorizations
$F_{s,N}=GH$ in which $G$ is monic of degree $r$ and $H$ is monic of degree $N-r$.

\begin{theorem}[Sparse factorization cover]\label{thm:canonical-factorization}\label{cor:canonical-Hilbert}
The homomorphism
\begin{equation}\label{eq:sparse-base-map}
 B_{s,N}\longrightarrow R_r/J_{I_{s,N}}^{(s)},
 \qquad t_j\longmapsto q_{j,N},
\end{equation}
induces a natural isomorphism
\begin{equation}\label{eq:factorization-isomorphism}
 R_r/J_{I_{s,N}}^{(s)}\simeq\Fact_r(F_{s,N}/B_{s,N}).
\end{equation}
As a $B_{s,N}$-module this algebra is free with basis
\begin{equation}\label{eq:factorization-Schur-basis}
 \{S_\lambda:\lambda\subseteq(N-r)^r\}.
\end{equation}
Consequently, the structural morphism
\begin{equation}\label{eq:canonical-pi}
 \pi_{s,N}:X_{s,N}\longrightarrow\A_{\mathbb Z}^{s-1}
\end{equation}
is finite flat of degree $\binom Nr$; it is the restriction of the incidence
projection $q$ to the monic relation chart.  Its root-weighted Hilbert series is

\begin{equation}\label{eq:canonical-Hilbert}
 \Hilb_{\mathcal O(X_{s,N})}(z)
 =\frac{\genfrac{[}{]}{0pt}{}{N}{r}_z}
 {\prod_{j=0}^{s-2}(1-z^{N-j})}
 =\frac{\prod_{j=N-r+1}^{N-s+1}(1-z^j)}
 {\prod_{j=1}^r(1-z^j)}.
\end{equation}
Here $\genfrac{[}{]}{0pt}{}{N}{r}_z$ is the Gaussian binomial coefficient.
The Hilbert series is computed using the $\mathbb Z$-ranks
of the root-weighted homogeneous components; equivalently,
it is the Hilbert series after base change to any field.
\end{theorem}

\begin{proof}
Modulo the universal $G$, the remainder of $F_{s,N}$ is
\begin{equation*}
 \sum_{j=0}^{s-2}(q_{j,N}-t_j)x^j+
 \sum_{b=s-1}^{r-1}q_{b,N}x^b.
\end{equation*}
Thus the universal factorization equations eliminate the variables $t_j$ and leave precisely
$(q_{s-1,N},\ldots,q_{r-1,N})=J_{I_{s,N}}^{(s)}$.  This proves
\eqref{eq:factorization-isomorphism}, functorially over every coefficient base.
The integral basis recalled above specializes to
\eqref{eq:factorization-Schur-basis}; its rectangle has $\binom Nr$ partitions.
The base variable $t_j=q_{j,N}$ has root weight $N-j$, while the Schur basis element
$S_\lambda$ has weight $|\lambda|$.  The generating function of the basis is the Gaussian
binomial.  Cancelling its final $s-1$ numerator factors gives the complete-intersection form.
\end{proof}

In the monic-divisor incidence $Z_{I,K}$, independence of
$Q_0,\ldots,Q_{s-2}$ forces the relation's $x^N$ coefficient to be
invertible.  Normalizing it gives $t_j=q_{j,N}$ scheme-theoretically.
Thus $Z_{I,K}$ is this factorization chart; the binary projective
incidence $\overline{\mathfrak Z}_I$ also contains its boundary.

\subsection{Grassmannian zero fiber}

The formal profile model in Proposition~\ref{prop:formal-profile-model} uses the
following specialization of the canonical factorization cover.

\begin{proposition}[Grassmannian zero fiber]\label{prop:grassmannian-zero-fiber}
The zero fiber of \eqref{eq:canonical-pi} has coordinate ring
\begin{equation}\label{eq:cohomology-fiber}
 \mathcal O(\pi_{s,N}^{-1}(0))
 =\mathbb Z[e_1,\ldots,e_r]/(h_{N-r+1},\ldots,h_N)
 \simeq H^*(\Gr(r,\mathbb C^N),\mathbb Z),
\end{equation}
where $e_i=c_i(\mathcal S^\vee)$ and cohomological degree is twice root
weight.  After base change to a field, this is an Artin local scheme of length
$\binom Nr$ supported at $G=x^r$.
\end{proposition}

\begin{proof}
The zero section adds $q_{0,N},\ldots,q_{s-2,N}$.  Lemma~\ref{lem:hook-remainder}
and the signed Jacobi--Trudi change in \eqref{eq:hook-triangular}, applied
through hook length $r-1$, turn the full remainder ideal into
$(h_{N-r+1},\ldots,h_N)$.  The Borel presentation gives
\eqref{eq:cohomology-fiber}; the Schur basis has $\binom Nr$ elements and
$x^r$ is the only monic degree-$r$ divisor of $x^N$ over an algebraically
closed field.  This gives the length and support.
\end{proof}

\section*{Funding}

This work was supported by the National Natural Science Foundation of China (grant numbers
12571003 and 12501006) and the Basic and Applied Basic Research Foundation of Guangdong
Province (grant number 2024A1515010589).

\section*{Declaration of competing interest}

The author declares that there are no known competing financial interests or personal
relationships that could have appeared to influence the work reported in this paper.

\section*{Data availability}

No data were used for the research described in this article.

\end{document}